\documentclass[aos,preprint]{imsart}

\RequirePackage[OT1]{fontenc}
\RequirePackage{amsthm,amsmath}
\RequirePackage[numbers]{natbib}

\startlocaldefs
\theoremstyle{plain}
\newtheorem{thm}{Theorem}
\newtheorem*{thm*}{Theorem}

\newtheorem{prop}{Proposition}
\newtheorem{lem}{Lemma}
\newtheorem*{lem*}{Lemma}

\newtheorem*{defi*}{Definition}
\newtheorem{ass}{Assumption}
\newtheorem{rque}{Remark}

\newtheorem*{rques*}{Remarks}

\endlocaldefs
\usepackage[dvipsnames]{xcolor}

\usepackage{todonotes}

\usepackage{enumitem}

\usepackage{amsfonts,amssymb,amsthm,amsmath}
\usepackage{stmaryrd}
\usepackage{makecell}
\usepackage{capt-of}
\usepackage{dsfont}
\usepackage{multirow}
\usepackage{url}

\allowdisplaybreaks

\usepackage{xifthen}
\usepackage{xargs}

\usepackage{mathtools}

\usepackage{longtable}
\usepackage{tabularx}

\usepackage{calrsfs}

\usepackage{array}
\usepackage{makecell}

\begin{document}

\begin{frontmatter}
\title{Scaled minimax optimality in high-dimensional linear regression: A non-convex algorithmic regularization approach}
\runtitle{ Adaptive Iterative Hard Thresholding}

\begin{aug}
\author{\fnms{Mohamed} \snm{Ndaoud}
\ead[label=e1]{ndaoud@usc.edu}}


\affiliation{University of Southern California
}

\address{Department of Mathematics\\
University of Southern California\\
Los Angeles, CA 90089 \\
\printead{e1}}

\end{aug}

\begin{abstract}
 The question of fast convergence in the classical problem of high dimensional linear regression has been extensively studied. Arguably, one of the fastest procedures in practice is Iterative Hard Thresholding (IHT). Still, IHT relies strongly on the knowledge of the true sparsity parameter $s$. In this paper, we present a novel fast procedure for estimation in the high dimensional linear regression. Taking advantage of the interplay between estimation, support recovery and optimization we achieve both optimal statistical accuracy and fast convergence. The main advantage of our procedure is that it is fully adaptive, making it more practical than state of the art IHT methods. Our procedure achieves optimal statistical accuracy faster than, for instance, classical algorithms for the Lasso. Moreover, we establish sharp optimal results for both estimation and support recovery. As a consequence, we present a new iterative hard thresholding algorithm for high dimensional linear regression that is scaled minimax optimal (achieves the estimation error of the oracle that knows the sparsity pattern if possible), fast and adaptive. 
\end{abstract}
\end{frontmatter}
 \section{Introduction}

Datasets with large numbers of features are becoming increasingly available and
important in every field of research and innovation. The representation of such data in any
coordinate system leads to so called high dimensional data, whose analysis is often associated
with phenomena that go beyond classical estimation theory. Further assumptions on the structure of the underlying signal are required in order
to make the estimation problem more well-defined. For instance in a problem of high dimensional
regression, we may assume that the vector to estimate is sparse. The present work aims to
develop an estimation method that can extract information from existing high dimensional structured datasets in a more efficient way. 
\subsection{Statement of the problem}

Assume that we observe the vector of measurements $Y\in\mathbf{R}^{n}$ satisfying 
\begin{equation}\label{eq:def}
     Y = X\beta + \sigma \xi ,
\end{equation}
where $X\in\mathbf{R}^{n \times p}$ is a given design or sensing matrix. The noise $\xi$ is a centered $1$-subGaussian random vector (i.e. $\forall \lambda \in \mathbf{R}^n, \mathbf{E}(e^{\langle \lambda , \xi \rangle}) \leq e^{ \| \lambda \|^2/2}$), and $\xi$ is independent of $X$. We denote by $\mathbf{P}_{\beta}$ the distribution of $(Y,X)$ in model \eqref{eq:def}, and by $\mathbf{E}_{\beta}$ the corresponding expectation.

For an integer $s\le p$, we assume that $\beta$ is $s$-sparse, that is it has at most $s$ non-zero components, and we denote by $S_{\beta}$ its support and $\eta_{\beta}$ the corresponding binary vector.  We also assume that components of $\beta$ cannot be arbitrarily small. This motivates us to define the following set $\Omega^{p}_{s,a}$ of $s$-sparse vectors: 
$$  \Omega_{s,a} = \left\{\beta \in \mathbf{R}^{p}:\quad  |\beta|_{0} \leq s \quad \text{and} \quad   |\beta_{i}| \geq a, \quad \forall i \in S_{\beta} \right\}, $$
where $a>0$, $\beta_{i}$ are the components of $\beta$ for $i=1,\dots ,p,$ and $|\beta|_{0}$ denotes the number of non-zero components of $\beta$. The value $a$ characterizes the scale of the signal. In the rest of the paper, we will always denote by $\beta$ the vector to estimate, while $\hat{\beta}$ will denote the corresponding estimator. Let us denote by $\psi$ the scaled minimax risk
\begin{equation}\label{eq:scaled_risk}
 \psi(s,a) =  \underset{\hat{\beta}}{\inf} \underset{\beta \in \Omega_{s,a}}{\sup}\mathbf{E}_{\beta}\left( \| \hat{\beta} - \beta\|^{2} \right), 
\end{equation}
where the infimum is taken over all possible estimators $\hat{\beta}$ and $\|.\|$ is the Euclidean norm. The risk \eqref{eq:scaled_risk} was introduced in \cite{ndaoud18} where sharp lower bounds were given under model \eqref{eq:def}. More precisely, assuming that the design columns are normalized (i.e. $\|X_i\| = \sqrt{n}$ for all $i=1,\dots,p$) and that the noise $\xi$ is a standard random Gaussian vector, a combination of theorems $3$ and $4$ in \cite{ndaoud18} provides the following sharp lower bounds:
\[
\psi(s,a) \geq (1+o(1))\frac{2\sigma^2s\log(ep/s)}{n}, \quad \forall a\leq(1-\varepsilon)\sigma\sqrt{\frac{2\log(ep/s)}{n} }
\]
and
\[
\psi(s,a) \geq (1+o(1))\frac{\sigma^2 s}{n}, \quad \forall a\geq(1+\varepsilon)\sigma\sqrt{\frac{2\log(ep/s)}{n} },
\]
that holds for any $0<\varepsilon \leq 1$ and such that the limit corresponds to $s/p \to 0$. One of the main motivations of the present work is to provide matching upper bounds for the risk $\psi$ under mild assumptions on the design matrix $X$.

{\bf Notation.}
In the rest of this paper we use the following notation. For any integer $n$, $[n]$ denotes the set of integers $\{1,\dots,n\}$. For given sequences $a_{n}$ and $b_n$, we say that $a_{n} = O(b_{n})$ (resp $a_{n} = \Omega(b_n)$) when $a_{n} \leq c b_{n}$ (resp $a_{n} \geq c b_{n}$) for some absolute constant $c>0$. We write $a_{n} \asymp b_{n} $ if $a_{n} = O(b_{n})$ and $a_{n}=\Omega(b_{n})$ while $a_n = o(b_n)$ corresponds to $a_n/b_n \to 0$ as $n$ goes to infinity. For ${\bf x},{\bf y}\in \mathbf{R}^{p}$, $\|{\bf x}\|_{\infty}$ is the $\ell_{\infty}$ norm of ${\bf x}$, $\|{\bf x}\|$ the Euclidean norm of ${\bf x}$, and $\langle{\bf x},{\bf y}\rangle$ the corresponding inner product. For a matrix $X \in \mathbf{R}^{n \times p}$, we denote by $X_j$ its $j$-th column, and $\|X\|_{2,\infty}:= \underset{j=1,\dots,p}{\max}\|X_{j}\|$. For $x,y \in \mathbf{R}$, we denote by $x\vee y$ the maximum of $x$ and $y$ and we set $x_{+}=x \vee 0$.
The notation  $\mathbf{1}\{\cdot\}$ stands for the indicator function. For a vector $Z \in \mathbf{R}^{p}$, $Z_{(i)}$ denotes the $i$-th non increasing order statistic of $Z$ such that $Z_{(1)}\geq \dots\geq Z_{(p)}$. For any finite set $S$, $|S|$ stands for its length. For any $X \in \mathbf{R}^{p \times p}$ and $S,S' \in \{1,\dots,p\}$, $X_{S}$ will denote the submatrix of $X$ with columns indexed by $S$, and $X_{S'S}$ the submatrix of $X$ with columns indexed by $S$ and rows indexed by $S'$. For a vector $M \in \mathbf{R}^{p}$, $M_{S}$ is the restriction of $M$ to the set $S$. For an SDP matrix $A$, we denote by $\lambda_{\max}$(resp. $\lambda_{\min}$) the largest (resp. lowest) corresponding eigenvalue, and  by $\|A\|_{F}$ its Frobenius norm. $\mathbf{I}_{p}$ is the identity matrix in $\mathbf{R}^{p\times p}$. 

For the sake of readability of the results, we will assume the design columns to be normalized in the rest of this section , i.e. for all $i=1,\dots,p$ we have $\|X_i\| = \sqrt{n}$ .
\subsection{Related literature}\label{sec:22}
The literature on minimax sparse estimation in high-dimensional linear regression (for both random and orthogonal design) is very rich and its complete overview falls beyond the format of this paper. We mention here only some recent results close to our work. All sharp results are considered in the regime where $\frac{s}{p} \to 0$ and $s\log(ep/s)/n\to0$. 

\begin{enumerate}
    \item 
        \textit{Discrepancy
            between the minimax rate
            $2s\sigma^2\log(ep/s)/n$ and the oracle rate $\sigma^2s/n$.
        }
        In the last decade,
        the success of estimators in sparse linear regression
        has been characterized by the minimax
        rate $2s\sigma^2\log(ep/s)/n$, achieved for instance 
        by the popular Lasso.
        It is well known 
        by practitioners that Lasso suffers from non-negligible bias,
        and this bias is unavoidable and at least of order $\sigma^2\frac{s\log{(p/s)}}{n}$ (cf. \cite{bellec2018}). Other convex estimators,
        such as Slope, suffer the same bias issue
        as the bias lower bound in \cite{bellec2018} applies as well.
        On the other hand, the oracle least-squares estimator
        restricted to the support of the true $\beta$
        enjoys the smaller rate $\sigma^2 s /n$.
        The focus of the present paper is the discrepancy
        between the well studied minimax rate $2\sigma^2\log(ep/s)/n$
        and the oracle rate $\sigma^2s/n$, in particular
        \begin{quotation}
            \textit{For which sparse $\beta$ is it possible to achieve
                the oracle rate $\sigma^2s/n$ without the knowledge
                of the true support ?
                When possible, how to construct estimators
                that achieve the oracle rate $\sigma^2s/n$?
            }
        \end{quotation}
        When the magnitude of entries of the signal $\beta$ is large enough we may expect to get a better estimate than the Lasso by removing the associated bias. In \cite{ndaoud18}, the separation at which the bias can be removed, in the sense that the oracle $\sigma^2s/n$ can be achieved,
        is exactly characterized in the case of orthogonal design.
        This separation is given by the universal separation $a^* = \sigma\sqrt{\frac{2\log{(p/s)}}{n}}$. In particular, for $a$ larger than $a^*$, the estimation risk could be as small as $\sigma^2\frac{s}{n}$. For the case of general designs, the same paper shows that $a \geq a^{*}$ is necessary in order to remove the bias without providing corresponding sufficient conditions. The
        recent paper \cite{zadik} provide insight on the precise phase transition when the design is Gaussian and the sparse vector $\beta$ is binary; however the findings of the present paper
        reveal that $a=a^{*}$ is the threshold at which
        the transition between $2\sigma^2s\log(ep/s)/n$ and $\sigma^2s/n$
        occurs for the general class of sparse vectors.
    
        A popular approach to avoid the bias present in convex estimators 
        is Iterative Hard Thresholding (IHT) \cite{blum} and its variants. IHT is a non-convex counterpart of the gradient descent corresponding to the Lasso \cite{shen2017tight,yuan2018gradient,zhou2018,wang2015}. It is shown in \cite{yuan2016exact,yuan2018gradient} that debaising and support recovery is possible,  using Gradient Hard Thresholding Pursuit, under the sub-optimal condition $a \asymp \sigma \sqrt{\frac{s\log(ep/s)}{n}}$. Two other approaches for unbiased estimation are either through concave penalization \cite{zhang2010,zhang2017} or  forward and backward greedy algorithms \cite{foba}. 
        \begin{quotation}
        \textit{When the oracle rate $\sigma^2s/n$ is achievable
            information-theoretically,
            is it possible to achieve this rate in polynomial time under mild assumptions on the design?}
        \end{quotation}

    \item \textit{Fast convergence and adaptation.} IHT uses explicitly the sparsity parameter, since at each step it keeps the $s$ largest values of the gradient descent. In \cite{liu2018between} for instance, IHT is shown to achieve the minimax rate $\sigma^2 s \log(ep/s)/n$ 
        after a number of iterations of order $\log\left( \frac{n\|\beta\|^2}{s\log(ep/s)}\right)$
        which corresponds to the expected number of iterations under strong convexity assumptions on the loss. In \cite{fast}, algorithms to compute
        the Lasso are shown to converge geometrically under stronger conditions. FoBa \cite{foba} achieves unbiased estimation without the knowledge
        of $s$, but may take longer to converge. 
	This motivates the following question: 
        \begin{quotation}
            \textit{Is it possible to achieve the minimax rate
            $2\sigma^2s\log(ep/s)$ and the oracle rate $\sigma^2s/n$
        using fast iterative algorithms, without the knowledge of $s$?}
        \end{quotation}
	\item \textit{Sharpness.} Here we refer to sharpness
            for the problem of statistical convergence rates
            where the statistical accuracy matters up to
            exact multiplicative constants.
            It is inspired by statistical physics where phase transitions occur at some sharp threshold.
            The notion of sharp optimality is useful to compare algorithms in practice since the constants generally hide large values that may
            impact practical implementations. 
            In \cite{su2016}, SLOPE is shown to be sharply minimax optimal on the set of sparse vectors.
            Similarly, \cite{montanari} show sharp results of estimation for the debiased Lasso (with a post-processing step)
            that are adaptively optimal up to a logarithmic factor.
            Results in both \cite{candes} and \cite{montanari} hold, under the Gaussian design assumption, in the asymptotic where $s/p \to 0$ and $s\log(p)/n \to 0$. To the best of our knowledge, no sharp results for general designs that are not necessarily Gaussian are known. 
            \begin{quotation}
            \textit{Is it possible to extend sharp minimax results with beyond Gaussian designs,
        for instance to sub-Gaussian designs? }
            \end{quotation}
\end{enumerate}

In this paper, we shed some light on these issues. Specifically, we address the above questions in what follows.

\subsection{Main contribution}

The present work is mainly devoted to bridging the gap between statistical optimality and optimization in a sharp an adaptive way. The main novelty is a unified framework to analyze simultaneously estimation error, support recovery and optimization speed. 
Our objective is to build a procedure that would answer positively the  questions stated in Section \ref{sec:22}. The proposed method is an iterative hard thresholding algorithm where the threshold is updated at each step. For $\lambda>0$, we define the hard thresholding operator $\mathbf{T}_{\lambda}:\mathbf{R}^{p}\to\mathbf{R}^{p}$, such that
$$
\forall u \in \mathbf{R}^{p}, \forall j =1,\dots,p,\quad \mathbf{T}_{\lambda}(u)_{j} = u_{j}\mathbf{1}\{|u_{j}|\geq \lambda\} .
$$
We consider here a general class of IHT estimators. For a given sequence $(\lambda_{m})_{m}$ of positive numbers, we define the corresponding sequence of estimators $(\hat{\beta}^{m})_{m}$ such that $\hat{\beta}^{0}=0$ and for $m=1,2,\dots$
\begin{equation}
    \hat{\beta}^{m} = \mathbf{T}_{\lambda_{m}}\left(\hat{\beta}^{m-1} + \frac{1}{n}X^{\top}(Y-X\hat{\beta}^{m-1})\right).
\end{equation}
This procedure corresponds to a projected gradient descent on a non-convex set. Our thresholding procedure is, in some sense, an interpolation between the two classical thresholds, namely the largest $s$ component for IHT, and $\sigma\sqrt{\frac{2\log(p)}{n}}$ for the LASSO. We start with a large threshold, we then update it geometrically until hitting the statistical universal threshold. This gives further an explicit stopping time of our procedure that may be seen as an early stopping rule for gradient descent. Another perspective about our procedure is that it could be seen as what we describe later as an \textit{iteration selection} procedure. At each step we may see $\hat{\beta}^{m}$ as an estimator computed using one iteration starting from the one before $\hat{\beta}^{m-1}$. By analogy with a classical model selection criterion, we can select the iteration that is minimax optimal. This leads, in particular, to a faster adaptive procedure compared to model selection since each estimator is simply one iteration of our algorithm.  Our contribution can be summarized as follows:
\begin{itemize}
\item We derive a new proof strategy to construct adaptive minimax optimal estimators through algorithmic regularization. 
    This combines techniques from non-convex optimization and model selection. 
\item We propose a fully adaptive variant of IHT that is scaled minimax optimal (i.e. achieves optimal risk of the oracle that knows the sparsity pattern when possible). We also show optimal support recovery results for this procedure. To the best of our knowledge, our conditions improve upon the previously known ones to achieve support recovery for an IHT procedure. When $s\log(p)^3 / n \to 0$, optimal conditions for the problem of support recovery in high dimensional linear regression under Gaussian design are provided in \cite{gao2019iterative} using an iterative procedure without sample splitting. Similarly, our methodology does not require sample splitting. Moreover it achieves support recovery for a larger class of designs under a milder condition.
\item We establish sharp optimal results under RIP as $\delta \to 0$. This in particular holds for sub-Gaussian designs as $s\log(ep/s)/n \to 0$. To the best of our knowledge, those are the first sharp estimation results to hold beyond Gaussian  design.
\item As for the optimization part, we use local strong convexity/local smoothness of the loss function (equivalent in our setting to RIP) in order to get fast global convergence as in \cite{fast} for convex penalized estimators.  Our analysis makes it possible to study algorithms with non-convex penalization for instance the hard thresholding penalization. Using statistical properties of the model, we benefit from both local strong convexity and non-convex penalization 
    in order to
    provide optimal worst-case computational guarantees of our procedure. 
\end{itemize}
As a consequence of our methodology, we extend results of scaled minimax optimality to the regression model under RIP. In particular we close the gap by showing that
\[
\psi(s,a) = (1+o(1))\frac{2\sigma^2s\log(ep/s)}{n}, \quad \forall a\leq(1-\varepsilon)\sigma\sqrt{\frac{2\log(ep/s)}{n} }
\]
and
\[
\psi(s,a) = (1+o(1))\frac{\sigma^2 s}{n}, \quad \forall a\geq(1+\varepsilon)\sigma\sqrt{\frac{2\log(ep/s)}{n} },
\]
that holds for any $\varepsilon>0$ and such that the limit corresponds to $s/p \to 0$ and $\delta \to 0$. Moreover, the upper bound is achieved through a polynomial time method that is fully adaptive.
 Some interesting questions arise based on adaptation to parameters on both optimization and statistical sides that we address in the Conclusion. We summarize our contribution to the problem of minimax scaled sparse estimation below.
 \begin{center}
\scalebox{1}{\begin{tabular}{|c|c|c|c|}
    \hline 
     & & & \\
     & \makecell{Non asymptotic \\ results}  & \makecell{Sharp results \\ ($a \leq (1-\varepsilon) a^{*}$)} & \makecell{Sharp results \\ ($a \geq (1+\varepsilon)a^{*}$)} \\
     & & & \\
   \hline 
        & & & \\
Minimax lower bounds
& \makecell{$C_4\sigma^2\frac{s\log(ep/s)}{n}$ ~~ \cite{bellec}}
& \makecell{$2\sigma^2\frac{s\log(ep/s)}{n}$ ~~ \cite{ndaoud18}}
& \makecell{$\sigma^2\frac{s}{n}$ ~~ \cite{ndaoud18}} \\
     & & & \\
     \hline 
        & & & \\
   \makecell{Risk of LASSO \\ (not adaptive to sparsity)} &
   \makecell{$C_1\sigma^2\frac{s\log(ep/s)}{n}$ ~~ \cite{bellec}} 
   & \makecell{$2\sigma^2\frac{s\log(ep/s)}{n}$ ~~ \cite{montanari}}
& \makecell{$2\sigma^2\frac{s\log(ep/s)}{n}$ ~~ \cite{bellec2018}} \\
     & & & \\
\hline 
     & & & \\
\makecell{Risk of SLOPE \\ (adaptive to sparsity)}
& \makecell{$C_2\sigma^2\frac{s\log(ep/s)}{n}$ ~~ \cite{bellec}}
& \makecell{$2\sigma^2\frac{s\log(ep/s)}{n}$ ~~ \cite{candes}}
& \makecell{$2\sigma^2\frac{s\log(ep/s)}{n}$ ~~ \cite{bellec2018}} \\
     & & & \\
\hline 
     & & & \\
RISK of adaptive IHT
& \makecell{$C_3\sigma^2\frac{s\log(ep/s)}{n}$ \\ 
    {\small This paper, Theorem \ref{thm:optimal:statistic:adap2}} }
& \makecell{$2\sigma^2\frac{s\log(ep/s)}{n}$ \\ 
{\small this paper, Theorem \ref{thm:sharp:1}}
}
& \makecell{$\sigma^2\frac{s}{n}$ 
   \\ {\small this paper, Theorem \ref{thm:sharp:2}}
} \\
     & & & \\
     \hline 
\end{tabular}}
\captionof{table}{Summary of minimax upper and lower bounds for estimation in high dimensional linear regression where $a^{*} = \sigma\sqrt{\frac{2\log(ep/s)}{n}}$ and $C_1,C_2,C_3,C_4>0$ some absolute constants. Sharp results hold for any $0<\varepsilon<1$.}
\end{center}

\section{Non-asymptotic minimax sparse estimation: A new proof strategy}\label{sec:3}
Classical non-asymptotic minimax results for sparse estimation in linear regression are proved for minimizers of well defined objective loss functions. In this section, we present and analyze our variant of iterative hard thresholding algorithm, and show similar minimax results.  In what follows we assume that the design $X$ satisfies the following condition. For an integer $s =1,\dots,p$, define $L_{s},m_{s}>0$ such that
$$
L_{s}  = \underset{|S|=s}{\max} \lambda_{\max}( X_{S}^{\top}X_{S}),
$$
and
$$
m_{s}  = \underset{|S|=s}{\min} \lambda_{\min} \left( X_{S}^{\top}X_{S} \right).
$$
Set  $\delta_{s}:= 1-\frac{m_{s}}{L_{s}}$. 
\begin{ass}\label{ass:1}
For $0<c<1$ and $s\in[p]$, we say that $X$ satisfies RIP($s,c$) if 
$$
\delta_{s} \leq c.
$$
\end{ass}
In the rest of the paper, we assume that $s\leq p/3$ and that $X$ satisfies RIP($3s$,$\delta/2$) for some $0 \leq \delta<1$. Assumption \ref{ass:1} is equivalent to Restricted Strong Convexity and Restricted Smoothness on the set of $s$-sparse vectors, where we assume that $\gamma_s:= \frac{L_s}{m_s}$, the \textit{condition number}, is bounded. Here are few remarks concerning this assumption with respect to adaptation. \begin{itemize}
    \item Although our results should hold for general $\gamma_s$, we decided to consider only the case of bounded $\gamma_s$, and omit the dependence on $\gamma_s$, to make the presentation of our results simpler and also since a fully adaptive procedure would require an upper bound on $\gamma_s$ . 
    \item Our results hold under a relaxed assumption on the design where there exits some matrix $M$ such that $MX^{\top}X$ satisfies a condition similar to RIP as in \cite{montanari}. For instance, in the case of general Gaussian design with full rank covariance $\Sigma$, and for $M$ chosen to be $\Sigma^{-1}$, the condition would hold under the usual assumption $s\log(ep/s) = O(n)$. Again, this requires knowing $M$ in advance and constrains adaptation.
    \item For gradient descent step in IHT, the step size depends on $\gamma_s$ or more precisely on $L_s$. When no upper bound on $\gamma_s$ is unknown, there exist adaptive choices of the step size based on the exact line search for instance.  
\end{itemize}  
For all above reasons, we decided to only focus on adaptivity with respect to statistical parameters of the problem, namely $s, \sigma$ and $\|\beta\|$, and to leave the general case with other results for further research. 

 Our procedure is an iterative hard thresholding algorithm where the threshold is updated at each step. For $\lambda>0$, define the hard thresholding operator $\mathbf{T}_{\lambda}:\mathbf{R}^{p}\to\mathbf{R}^{p}$, such that
$$
\forall u \in \mathbf{R}^{p}, \forall j =1,\dots,p,\quad \mathbf{T}_{\lambda}(u)_{j} = u_{j}\mathbf{1}\{|u_{j}|\geq \lambda\} .
$$
Notice that the usual IHT algorithm corresponds to $\mathbf{T}_{u_{(s)}}(u)$ where $u_{(s)}$ is the $s$ largest entry of $u$, hence the threshold is data-dependent but most importantly sparsity dependent \cite{blum}. For a given sequence $(\lambda_{m})_{m}$ of positive numbers, we define the corresponding sequence of estimators $(\hat{\beta}^{m})_{m}$ such that $\hat{\beta}^{0}=0$ and for $m=1,2,\dots$
\begin{equation}\label{eq:estimator_non_adap}
    \hat{\beta}^{m} = \mathbf{T}_{\lambda_{m}}\left(\hat{\beta}^{m-1} + \frac{1}{\|X\|_{2,\infty}^{2}}X^{\top}(Y-X\hat{\beta}^{m-1})\right).
\end{equation}
This procedure corresponds to a projected gradient descent on a non-convex set. The choice of normalizing the gradient by $\|X\|_{2,\infty}$ instead of $L_{3s}$ is due to the fact that $L_{3s}$ is not tractable and that we do not consider adaptivity with respect to optimization parameters $(L_{3s},\delta_{3s})$ in this work. If $\delta_{3s}$ is small enough, it is easy to see that $\|X\|_{2,\infty}$ is a good proxy for $L_{3s}$. The usual projection step consists in keeping the largest $s$ components. The operator $\mathbf{T}_{\lambda}$ plays a similar role here. Imposing sparsity at each step of the procedure is crucial in order to benefit from restricted properties of the design. The novelty of our procedure lies in the fact that it implicitly grants the sparsity of our estimator at each step without having to choose exactly $s$ components.

Unlike the analysis of IHT, in previous works, that benefits from local convex properties of the objective function, we choose to directly analyze the non-convex gradient descent algorithm and leverage the structure of both the signal and design in order to get a contraction of the error. We give here the intuition behind our procedure. In what follows we propose a specific choice for the sequence of thresholds $(\lambda_{m})_{m}$ that will allow us to achieve both optimal statistical accuracy and fast convergence of the algorithm in an adaptive way. Let $\lambda_0, \lambda_\infty>$ and $0<\kappa<1$ be given constants, we define the sequence $(\lambda_{m})_{m}$ as follows 
\begin{equation}\label{eq:def:threshold}
    \lambda_{m} = \kappa^{m/2} \lambda_0 \vee \lambda_\infty,\quad m=0,1,\dots.
\end{equation}
The sequence of thresholds starts at some very large threshold $\lambda_0$, then keeps updating it linearly until reaching a final threshold given by $\lambda_\infty$. A good choice of $\lambda_\infty$ is given by the universal statistical threshold $\sqrt{\frac{2\sigma^2 \log(ep/s)}{\|X\|^2_{2,\infty}}}$. Our choice of the thresholding sequence is motivated by the following. Observe that
\begin{align}\label{eq:simple}
\hat{\beta}^{m} &+ \frac{1}{\|X\|_{2,\infty}^{2}}X^{\top}(Y-X\hat{\beta}^{m}) =\\
&\beta + \underbrace{\left(\frac{1}{\|X\|_{2,\infty}^{2}}X^{\top}X - \mathbf{I}_{p} \right)(\beta-\hat{\beta}^{m})}_{\text{optimization error}}
+ \underbrace{\frac{\sigma}{\|X\|_{2,\infty}^{2}}X^{\top}\xi}_{\text{statistical error}}. \notag
\end{align}
At each step, we can decompose the estimation error into two parts. An optimization error that may be reduced thanks to the local contraction of the design (RIP), and a statistical error that is unavoidable. While it is well understood that the choice of a threshold of order $\sqrt{\frac{2\sigma^2\log(ep/s)}{\|X\|^2_{2,\infty}}}$ is optimal in order to control the statistical error, the same threshold does not grant sparsity of the estimator at first steps. In fact, if the signal $\beta$ is ``well-spread" (i.e. its coordinates share similar magnitude) and $\|\beta\|$ large enough, it may occur that we select too many coordinates at the first step. This lack of sparsity makes it hard to benefit from restricted properties of the design. The alternative choice of a very large threshold grants sparsity but leads to a high statistical error. The usual choice of keeping the largest $s$-components at each step, is a natural fix. This intuition is similar to the motivation behind the LARS algorithm \cite{lars}. Our thresholding procedure is, in some sense, an interpolation between the two classical thresholds, namely the largest $s$-component for IHT, and $\sqrt{\frac{2\sigma^2\log(ep/s)}{\|X\|^2_{2,\infty}}}$ for LASSO. We start with a threshold large enough, then as we move forward the optimization error gets smaller which allows us to update the threshold without loosing the contraction. Our choice of thresholding sequence gives also an explicit stopping time, as we may stop once the threshold hits the universal statistical threshold. In the rest of the paper, we use the following notation:
\[
     \Xi:=\frac{\sigma}{\|X\|_{2,\infty}^{2}}X^{\top}\xi \quad \text{and} \quad 
   \Phi:= \left(\frac{1}{\|X\|^{2}_{2,\infty}}X^{\top}X - \mathbf{I}_{p} \right).
\]
It is useful to observe that as long as $X$ satisfies  RIP($3s$,$\delta/2$), then $\Phi$ is a contraction for $3s$-sparse vectors. This is rephrased in the following Lemma that we prove in the Appendix.
\begin{lem}\label{lem:contraction}
If $X$ satisfies RIP($3s$,$\delta/2$) then for all $S \subset \{1,\dots,p\}$ such that $|S|\leq 3s$, we have $\lambda_{\max}(\Phi_{SS})\leq \delta$.
\end{lem}
Before stating our results we give a first result that is relevant to the analysis of our algorithm. We draw the reader's attention, that our analysis is fully deterministic since we place ourselves in a well chosen random event that captures the complexity of our model. Namely, we consider the event 
$$
\mathcal{O} = \left\{ \sum_{i=1}^{s}\Xi_{(i)}^{2} \leq \frac{10\sigma^{2}s\log(ep/s)}{\|X\|_{2,\infty}^{2}} \right\}.
$$
Using Lemma \ref{lem:order:statistic} (cf. Appendix), the event $\mathcal{O}$ holds with high probability. Conditionally on the event $\mathcal{O}$, the next Theorem shows that, at each step, the corresponding estimator is $2s$-sparse and the surrogate function of the estimation error, given by $s\lambda_m^2$, decreases exponentially.   
\begin{thm}\label{thm:main}
Assume that $\beta$ is $s$-sparse, that is $|\beta|_{0}\leq s$ and that $X$ satisfies RIP($3s$,$\delta/2$). We denote by $S$ the support of $\beta$. Let $\lambda_0,\lambda_\infty>0$, $0<\kappa<1$, and define $(\hat{\beta}^m)_m$ and its corresponding thresholding sequence $(\lambda_{m})_{m}$ as in \eqref{eq:estimator_non_adap}-\eqref{eq:def:threshold}. Assume that $\delta < 1/36 \vee \kappa$, $\|\beta\| \leq \sqrt{s} \lambda_0$ and $\frac{\sigma\sqrt{40\log(ep/s)}}{\|X\|_{2,\infty}} \leq \lambda_{\infty}$. If $\mathcal{O}$ holds, then for all $m$, we have
\begin{equation}\label{eq:main:1}
    |\hat{\beta}^{m}_{S^{c}}|_{0} \leq s,
\end{equation}
and
\begin{equation}\label{eq:main:2}
    \|\hat{\beta}^{m} - \beta\|^{2} \leq 9s\lambda_{m}^{2}.
\end{equation}
\end{thm}
Unlike for  convex  regularized least squares, the objective function $\| \hat{\beta}^m - \beta\|^2$ does not decrease at each step. Indeed, our gradient descent step may be trapped in some saddle points for instance when there is a big gap between coordinates of $\beta$. We get around this issue by finding an upper bounding surrogate function that decreases exponentially. This can be viewed as the highlight of our non-convex approach. The sequence $\lambda_{m}$ decreases until it reaches the stationary threshold $\lambda_{\infty}$. If we tune $\lambda_\infty$ with the universal statistical threshold then the final estimation error is optimal. In that case, we may stop the algorithm after a number of steps of order $\log\left(\frac{s\lambda_0^{2}\|X\|_{2,\infty}^{2}}{\sigma^{2} s\log(ep/s)}\vee2\right) / \log(1/\kappa) $. Notice that we do not need to know the precise value of $\delta$ here but only an upper bound, that we set to $1/36$ in Theorem \ref{thm:main}. We did not try to optimize this value. We may also set $\lambda_0$ as large as possible and our algorithm would still output an estimator that attains the minimax optimal rate. If $s\lambda_0^{2}$ is much larger than $\|\beta\|^{2}$, then the result of Theorem \ref{thm:main} is not optimal from an optimization perspective, in the sense that it would take more steps to stop compared to IHT for instance. In order to achieve both optimal statistical accuracy and optimal fast convergence, we need $\lambda_0$ to be roughly of the same order as $\frac{\|\beta\|}{\sqrt{s}} \vee \sigma\frac{\sqrt{\log(ep/s)}}{\|X\|_{2,\infty}}$. The optimal choices of $\lambda_0$ and the stopping time $m$ depend on $\|\beta\|$. In order to derive minimax optimal results, we need to make these choices adaptive with respect to $\|\beta\|$.
We show next how to tune $\lambda_0$ and $m$ in order to grant linear convergence of our procedure. Denote by $M$ the vector 
$$M:=\frac{1}{\|X\|_{2,\infty}^{2}}X^{\top}Y = \beta + \Phi\beta + \Xi,
$$
and set 
\begin{equation}\label{eq:threshold:initial}
    \hat{\lambda}_0 =\sqrt{\frac{10\sum_{i=1}^{s}M_{(i)}^{2}}{s}} \vee\frac{\sigma}{\|X\|_{2,\infty}} \sqrt{40\log(ep/s)},
\end{equation}
and
\begin{equation}\label{eq:stop}
    \hat{m} = \left\lfloor  2\log\left(\frac{ \hat{\lambda}_0^{2} \|X\|_{2,\infty}^{2}}{ 40\sigma^{2}\log(ep/s)} \right)/\log(1/\kappa) \right\rfloor+1.
\end{equation}
\begin{prop}\label{prop:threshold:initial}
Let $\beta$ be $s$-sparse and let $X$ satisfy RIP($3s$,$\delta/2$). Assume that $\delta \leq 1/4$ and that event $\mathcal{O}$ holds. Then
$$
 \left(\|\beta\|\vee \frac{\sigma\sqrt{10s\log(ep/s)}}{\|X\|_{2,\infty}}\right) \leq \sqrt{s}\hat{\lambda}_0  \leq 10\left(\|\beta\|\vee \frac{\sigma\sqrt{10s\log(ep/s)}}{\|X\|_{2,\infty}}\right),
$$
and
$$
  2\log\left(\frac{ \|\beta\|^2 \|X\|_{2,\infty}^{2}}{ 40\sigma^{2}s\log(ep/s)}\vee 1/4 \right)/\log(1/\kappa)  \leq \hat{m} - 1 \leq 2\log\left(\frac{ 5\|\beta\|^2 \|X\|_{2,\infty}^{2}}{ \sigma^{2}s\log(ep/s)}\vee 25 \right)/\log(1/\kappa).
$$
\end{prop}
Observe that for large values of $\|\beta\|$, $\hat{\lambda}_0$ is of the same order as $\frac{\|\beta\|}{\sqrt{s}}$ with high probability and that $\hat{m}$ is of order $\log\left(\frac{ \|\beta\|^2 \|X\|_{2,\infty}^{2}}{ \sigma^{2}s\log(ep/s)} \right)/\log(1/\kappa)$. We are now ready to state a result of the minimax optimality of our procedure. We will also pick $\lambda_\infty$ to be of the same order as the universal threshold
\begin{equation}\label{eq:lamda:unviersal}
\hat{\lambda}_{\infty} = \frac{\sigma\sqrt{40\log(ep/s)}}{\|X\|_{2,\infty}}.
\end{equation}
\begin{thm}\label{thm:optimal:statistic}
Let $0<\kappa<1$. Assume that $\delta \leq 1/36 \vee \kappa$ and that $X$ satisfies RIP($3s$,$\delta/2$). Let $\hat{\lambda}_0$ and $\hat{\lambda}_\infty$ given by \eqref{eq:threshold:initial} and \eqref{eq:lamda:unviersal}, $(\lambda_{m})_{m}$ be the corresponding sequence of estimators \eqref{eq:def:threshold}, and $\hat{m}$ be the stopping time \eqref{eq:stop}. Then the following holds
$$
\underset{|\beta|_{0} \leq s}{\sup}\mathbf{P}_{\beta}\left( \| \hat{\beta}^{\hat{m}} - \beta \|^{2} \geq \frac{360\sigma^{2}s\log(ep/s) }{\|X\|_{2,\infty}^{2}} \right) \leq e^{-c_{1}s\log(ep/s)},
$$
$$
\underset{|\beta|_{0} \leq s}{\sup}\mathbf{P}_{\beta}\left( |\hat{\beta}^{\hat{m}}|_0\geq 2s\right) \leq e^{-c_{1}s\log(ep/s)},
$$

and 
\begin{align*}
    \underset{|\beta|_{0} \leq s}{\sup}\mathbf{P}_{\beta}&\left( \hat{m} \geq  2\log\left(\frac{5\|\beta\|^{2}\|X\|_{2,\infty}^{2}}{\sigma^{2}s\log(ep/s)} \vee 25\right) /\log{(1/\kappa)} +1\right)\\
    &\leq e^{-c_{1}s\log(ep/s)},
\end{align*}

for some absolute constant $c_{1}>0$.
\end{thm}
Theorem \ref{thm:optimal:statistic} shows that $\hat{\beta}^{\hat{m}}$ achieves optimal statistical accuracy in linear time. Notice that $\hat{m}$ depends on $\log(1/\kappa)$ instead of $\log(1/\delta_{3s})$ simply because $\delta_{3s}$ is intractable. As we emphasized earlier, optimal results that are adaptive to $\gamma_{3s}$ (or equivalently $\delta_{3s}$) fall beyond the scope of this paper.  Theorem \ref{thm:optimal:statistic} shows minimax optimality of an estimator constructed through a non-convex algorithmic regularization scheme. Similar estimation error is also achieved using the SLOPE estimator as in \cite{bellec}. The advantage of the proposed estimator is that it achieves the optimal statistical accuracy in linear time.
\section{A fully adaptive minimax optimal procedure}
The above minimax estimator depends on the statistical parameters of the model $s$ and $\sigma$, and this only through the thresholding sequence $(\lambda_m)_m$. In particular, the choices of $\hat{\lambda}_0$, $\hat{\lambda}_\infty$, and $\hat{m}$ depend on $s$ and $\sigma$. For estimation of $\sigma$, we may consider the sequence of estimators $\hat{\sigma}_m^2$ of $\sigma^{2}$ such that
\begin{equation}\label{eq:sigma:initial}
    \hat{\sigma}_m^2 = \frac{\|Y-X\hat{\beta}^{m} \|^{2}}{n}.
\end{equation}
In practice, estimator \eqref{eq:sigma:initial} is considered in the Square-Root Lasso \cite{belloni2011square} among others in order to adapt to the noise level in high-dimensional regression. During the first steps, $\hat{\sigma}_m^2$ may have bad performance but in this case this means that the threshold $\lambda_m$ is much larger than the universal statistical threshold making precise estimation of $\sigma$ not necessary. As we get closer to the final threshold the estimation error gets smaller and estimation of $\sigma$ is improved as long as $n = \Omega(s\log(ep/s))$. The latter condition is sufficient in order to achieve good estimation of $\sigma$. It is worth saying that the same condition is not more restrictive than RIP. Indeed, a consequence of Corollary $7.2$ in \cite{bellec} implies that $n=\Omega(s\log(ep/s))$ as long as $X$ satisfies RIP. 

Concerning the choice of the initial threshold, recall that Theorem \ref{thm:optimal:statistic} hold if we replace $\hat{\lambda}_0$ by any upper bound, off to running more iterations. Hence, we can replace $\hat{\lambda}_0$ by the adaptive initial threshold
\begin{equation}\label{eq:threshold:adap}
    \bar{\lambda}_0 =\sqrt{20}|M|_{(1)} \vee\frac{\hat{\sigma}_0}{\|X\|_{2,\infty}} \sqrt{160\log(ep)}.
\end{equation}
Based on the fact that $u_{(1)} \geq \frac{1}{s}\sum_{i=1}^{s}u_{(i)}^2$,  threshold \eqref{eq:threshold:adap} is indeed an upper bound for the initial choice $\hat{\lambda}_0$ in \eqref{eq:threshold:initial}. Notice that the threshold $\bar{\lambda}_0$ in \eqref{eq:threshold:adap} can be as large as $\|\beta\|$ and not $\|\beta\|/\sqrt{s}$ as it was for $\hat{\lambda}_0$. This loss only appears in the number of iterations where we may run our algorithm for $\log(s)$ more steps. We present now two fully adaptive procedures that are minimax optimal.

\subsection{Adaptive early stopping}
We may now define a new adaptive thresholding sequence as follow 
\begin{equation}\label{eq:def:threshold:adap}
    \lambda_{m} = \kappa^{m/2} \lambda \vee \frac{\hat{\sigma}_m}{\|X\|_{2,\infty}}\sqrt{160\log(ep)},\quad m=0,1,\dots.
\end{equation}
 Our choice of $m$ is adaptive as well and given by
\begin{equation}\label{eq:stop:adap}
    \bar{m} = \inf\left\{ m / \lambda_m \leq \frac{\hat{\sigma}_m}{\|X\|_{2,\infty}}\sqrt{160\log(ep)} \right\} + 1.
\end{equation}
Observe that the adaptive stopping rule is exactly given by the step when we hit the statistical threshold. We can now state a minimax optimal result corresponding to our fully adaptive procedure. 
\begin{thm}\label{thm:optimal:statistic:adap}
Let $0<\kappa<1$. Assume that $\delta \leq 1/36 \vee \kappa$, that $n > 14000s\log(ep)$ and that $\mathbf{E}(\xi \xi^{\top}) = \mathbf{I}_n$. Let $\bar{\lambda}_0$ and $(\hat{\sigma}_{m})_m$ be defined as in \eqref{eq:threshold:adap} and \eqref{eq:sigma:initial},  $(\lambda_{m})_{m}$ be the corresponding sequence \eqref{eq:def:threshold:adap} and $\bar{m}$ be the stopping time \eqref{eq:stop:adap}. Then the following holds
$$
\underset{|\beta|_{0} \leq s}{\sup}\mathbf{P}_{\beta}\left( \| \hat{\beta}^{\bar{m}} - \beta \|^{2} \geq \frac{4000\sigma^{2}s\log(ep) }{\|X\|_{2,\infty}^{2}} \right) \leq e^{-c_{1}s\log(ep/s)},
$$
$$
\underset{|\beta|_{0} \leq s}{\sup}\mathbf{P}_{\beta}\left( |\hat{\beta}^{\bar{m}}|_0\geq 2s\right) \leq e^{-c_{1}s\log(ep/s)},
$$
and
\begin{align*}
    \underset{|\beta|_{0} \leq s}{\sup}\mathbf{P}_{\beta}&\left( \bar{m} \geq  2\log\left(\frac{10\|\beta\|^{2}\|X\|_{2,\infty}^{2}}{\sigma^{2}\log(ep)} \vee 100 \right) /\log{(1/\kappa)} +1\right)\\
    &\leq e^{-c_{1}s\log(ep/s)},
\end{align*}
for some absolute $c_{1}>0$.
\end{thm}
The adaptive procedure of Theorem \ref{thm:optimal:statistic:adap} is minimax optimal up to a logarithmic factor ( replacing $\log(ep/s)$ by $\log(ep)$). Conditions $n = \Omega(s\log(ep))$ and $\mathbf{E}(\xi \xi^{\top}) = \mathbf{I}_n$ are only required for adaptation to the noise level $\sigma$. Finally the number of iterations may be larger by $\log(s)$ compared to analogous non adaptive results. Hence, up to a logarithmic loss in both statistical accuracy and optimization speed, our early stopping procedure is fast, fully adaptive and minimax optimal.

\subsection{Iteration selection}
In order to capture the optimal dependence with respect to $s$, we rely on a different approach. One of the most popular methods to achieve adaptation is through the lens of model selection introduced in \cite{birge2001gaussian}. Given a set of models (or estimators) one picks a good estimator based on some criterion. More concretely, one may think of Cross-Validation where for each regularization parameter $\lambda$ an estimator is constructed, then the resulting estimators are either aggregated or one of them is chosen based on some criterion. In what follows, we present an adaptive procedure in the same flavor. We like to see our procedure as an iteration selection method instead of model selection.  Indeed, we can think of our $m-$th iteration as an estimator corresponding to $\lambda_m$. With this analogy in mind, penalized model selection boils down to a penalized iteration selection. Observe that iteration selection is much faster than the classical model selection since each estimator is computed using only one iteration initialized with the previous estimator. More concretely, we construct all iterations corresponding to the thresholding sequence 
\begin{equation}\label{threshold:adap:selection}
    \lambda_{m} = \kappa^{m/2} \bar{\lambda}_{0},\quad \forall m\in [\hat{T}],
\end{equation}
where $\bar{\lambda}_{0}$ was defined in \eqref{eq:threshold:adap} and $\hat{T}$ is defined below. The selected iteration is given by
\begin{equation}\label{eq:selection}
    \tilde{m} = \arg \underset{m\in[\hat{T}]}{\min}\left\{ \frac{1}{n}\|Y - X\hat{\beta}^m \|^2 + \frac{10\hat{\sigma}_{\bar{m}}^2 |\hat{\beta}^m|_0\log(ep/|\hat{\beta}^m|_0)}{n}\right\},
\end{equation}
where $\hat{\sigma}_{\bar{m}}$ was defined above. Again the problem of estimation of $\sigma$ is easier as long as $s\log(ep) = O(n)$ and we may replace $\hat{\sigma}_{\bar{m}}$ by any good estimator of $\sigma$. 
For completeness of our result we decided to stop the search domain over $m$ once the thresholding sequence $(\lambda_m)_{m}$ is below $\frac{\sigma}{\|X\|_{2,\infty}}$ where solutions are not granted to be sparse anymore. For that reason we set $\hat{T}$ such that
\[
\hat{T}= \inf\left\{ m \geq 0/  \lambda_m\leq \frac{4 \hat{\sigma}_{\bar{m}}}{\|X\|_{2,\infty}} \right\},
\]
where $(\lambda_m)_m$ is defined in \eqref{threshold:adap:selection}. We get the following result.
\begin{thm}\label{thm:optimal:statistic:adap2}
Let $0<\kappa<1$. Assume that $\delta \leq 1/36 \vee \kappa$ and that $n > 14000s\log(ep)$ and that $\mathbf{E}(\xi \xi^{\top}) = \mathbf{I}_n$. Let  $\tilde{m}$ be defined in \eqref{eq:selection}, then the following holds
$$
\underset{|\beta|_{0} \leq s}{\sup}\mathbf{P}_{\beta}\left( \| \hat{\beta}^{\tilde{m}} - \beta \|^{2} \geq \frac{100^2\sigma^{2}s\log(ep/s) }{\|X\|_{2,\infty}^{2}} \right) \leq e^{-c_{1}s\log(ep/s)},
$$
$$
\underset{|\beta|_{0} \leq s}{\sup}\mathbf{P}_{\beta}\left( |\hat{\beta}^{\tilde{m}}|_0\geq 3s\right) \leq e^{-c_{1}s\log(ep/s)},
$$
and
\begin{align*}
    \underset{|\beta|_{0} \leq s}{\sup}\mathbf{P}_{\beta}&\left( \hat{T} \geq  2\log\left(\frac{10\|\beta\|^{2}\|X\|_{2,\infty}^{2}}{\sigma^{2}} \vee 100 \right) /\log{(1/\kappa)} +1\right)\\
    &\leq e^{-c_{1}s\log(ep/s)},
\end{align*}
for some absolute $c_{1}>0$.
\end{thm}
The iteration selection procedure is different compared to the early stopping one since it chooses the best threshold $\lambda_m$ instead of worrying about tuning the stopping rule. Moreover it achieves the minimax optimal rate of $\sigma^2 s\log(ep/s)/\|X\|^2_{2,\infty}$ adaptively to all parameters under the mild condition $s\log(ep) = O(n)$.  Notice also that the number of constructed estimators $\hat{T}$ is small with overwhelming probability. Overall, full optimal adaptation on the statistics side comes with the price of $\log(s\log(ep/s))$ more steps on the optimization side. The our knowledge, Theorem \ref{thm:optimal:statistic:adap2} is the first to provide a fast and adaptive procedure that is minimax optimal.

\section{Sharp results and scaled minimax optimality}
In this section we present sharp minimax results for the problem of estimation in high dimensional linear regression. Moreover, we show that a variant of our estimator is scaled minimax optimal, improving upon regularized convex estimators that provably suffer from an unavoidable bias term. Our final procedure is an IHT algorithm with fixed threshold that is initialized by $\hat{\beta}^{\tilde{m}}$ defined earlier (Theorem \ref{thm:optimal:statistic:adap2}). By analogy with non-convex optimization, the step where the initialization is constructed plays the role of the first iterations getting to the basin of attraction.    
For a given $\lambda >0$, our final procedure $(\tilde{\beta}^m)_m$ is a variant of IHT  where $\tilde{\beta}^0 = \hat{\beta}^{\tilde{m}}$ and for all $m \geq 1$
\begin{equation}\label{eq:est:sharp}
      \tilde{\beta}^{m} = \mathbf{T}_{\lambda}\left(\tilde{\beta}^{m-1} + \frac{1}{\|X\|_{2,\infty}^{2}}X^{\top}(Y-X\tilde{\beta}^{m-1})\right).
\end{equation}
For any $\epsilon>0$, define the statistical threshold given by
\begin{equation}\label{eq:def:threshold:2}
    \lambda_{\infty}^{\epsilon} = (1+\sqrt{\epsilon}) \frac{\sigma \sqrt{2\log(ep/s)}}{\|X\|_{2,\infty}}.
\end{equation}
\begin{rque}
\begin{itemize}

    \item We can replace $\tilde{\beta}^{0}$ by any estimator that is minimax optimal and is at most $2s$- sparse. In particular one may choose to initialize our procedure with square-root slope \cite{derumigny2018improved}. Our choice of initialization is not only adaptive but is also fast to compute.
    \item Sharp adaptation to $\sigma$ can be achieved by choosing $\tilde{\sigma}^2:= \frac{1}{n}\|Y - X \hat{\beta}^{\tilde{m}}\|^2$. It is easy to observe that $\tilde{\sigma}=\sigma(1+o(1))$ with overwhelming probability under the mild assumption $s\log(ep)/n \to 0$. Hence, all our sharp results hold adaptively to $\sigma$ under the above condition.
\end{itemize}
\end{rque}
Sharp results are stated under the standard conditions $s/p \to 0$ and $\delta \to 0$ as $p \to \infty$. Indeed, the first condition is relevant to observe a strict change of behaviour between biased and unbiased estimators
and the second one corresponds to $s\log(ep/s)/n \to 0$ under Gaussian design. We remind the reader that $s,n,\delta$ depend on $p$ as $p \to \infty$.  Our next result states that one step is enough to achieve sharp optimal  minimax  estimation.
\begin{thm}\label{thm:sharp:1}
Let $\epsilon \in (0,1)$ and $\lambda_{\infty}^{\epsilon}$ given by \eqref{eq:def:threshold:2}. Assume that $\delta \leq  \epsilon \vee 1/400^2$. Let $\lambda \geq \lambda_{\infty}^{\epsilon}$ and let $\tilde{\beta}^{m}$ be the corresponding sequence of estimators \eqref{eq:est:sharp}. Then, for any $m \geq 0$ we have 
\begin{align*}
\underset{s/p \to 0}{\lim}\underset{|\beta|_{0}\leq s}{\sup}\mathbf{P}_{\beta}\Big( \| \tilde{\beta}^{m} - \beta \|
\geq
(1 + 4\sqrt{\delta} + 100 \delta^{m/2} +o(1))
\sqrt{s}\lambda \Big) = 0.
\end{align*}
\end{thm}
Theorem \ref{thm:sharp:1} is sharp in the sense that  when $\delta \to 0$, then for any $\epsilon>0$, there exists an estimator (depending on $\epsilon$) achieving the asymptotic minimax error of $(1+\epsilon)\frac{2\sigma^{2}s\log(ep/s)}{\|X\|_{2,\infty}^{2}}$ and that estimator corresponds to the threshold $\lambda = \lambda^{\epsilon}_{\infty}$. Replacing $\log(ep/s)$ by $\log(ep)$ in this choice of $ \lambda$ leads to an adaptive nearly sharp minimax optimal procedure. A similar result was shown for the SLOPE estimator in \cite{su2016} under Gaussian isotropic designs and in \cite{montanari} for more general Gaussian designs with known covariance.  Our sharp results do not require $\sigma$ to be known nor the design to be Gaussian, since we conduct a deterministic analysis over the design. 

Another advantage of our procedure, is that it eliminates the usual bias due to regularization under almost optimal conditions. Namely, as long as the informative signal components are well separated from zero, then our estimator achieves the same rate of estimating an $s$-sparse vector as if its support were known. 
The next result proves that our procedure is scaled minimax optimal.
\begin{thm}\label{thm:sharp:2}
Let $a>0$ and $\epsilon \in (0,1)$. Assume that conditions of Theorem \ref{thm:sharp:1} hold and that $s \to \infty$. If  $a \geq \lambda(1+\sqrt{\epsilon})$, then $\forall  m \geq \log(\log(ep/s))$ we have
\begin{align*}
\underset{s \to \infty, s/p \to 0}{\lim} \underset{\beta \in \Omega_{s,a}}{\sup}\mathbf{P}_{\beta}&\left( \| \tilde{\beta}^{m} - \beta \| \geq (1+4\sqrt{\delta}+o(1))\frac{\sigma^{2}s}{\|X\|_{2,\infty}^{2}} \right) = 0.
\end{align*}
\end{thm}
Notice first, that if $a = \Omega\left( \frac{\sigma\sqrt{\log(ep/s)}}{\|X\|_{2,\infty}} \right)$, then we can construct an estimator  $\tilde{\beta}^{m}$ achieving the non-asymptotic minimax parametric statistical error of order $\frac{\sigma^{2}s}{\|X\|_{2,\infty}^{2}}$.
For the sharp counterpart, observe that for any $\epsilon>0$, if $\delta \to 0$ and $s\to \infty$, then there exists an estimator $\tilde{\beta}^{m}$ (depending on $\epsilon$) that achieves the optimal error of $(1+o(1))\frac{\sigma^{2}s}{\|X\|_{2,\infty}^{2}}$ w.h.p under the condition $a \geq (1+\epsilon)\frac{\sigma\sqrt{2\log(ep/s)}}{\|X\|_{2,\infty}}$. The last condition on $a$ is necessary in order to achieve such a result as shown in \cite{ndaoud18}. This result shows that $\tilde{\beta}^{m}$ is sharply scaled minimax optimal as long as $m \geq \log(\log(ep/s))$. Again full adaptation is granted replacing $\log(ep/s)$ by $\log(ep)$ which leads to nearly sharp optimal results.
\section{On support recovery}
Throughout the paper our proofs are based on simultaneous analysis of both estimation error and variable selection. As a consequence, we can also recover results for support recovery. For a given vector $\beta$, we denote by $\eta$ the corresponding decoder i.e $\eta_{i} = \mathbf{1}(\beta_{i} \neq 0)$. We first give a straightforward result for almost full recovery , i.e $\frac{|\hat{\eta}-\eta|}{s} \to 0$, based on the previous section.
\begin{thm}\label{thm:sharp:3}
Under the conditions of Theorem \ref{thm:sharp:2}, we get, for all $m \geq \log(\log(ep/s))$, that
$$
 \underset{s/p \to 0}{\lim} \underset{\beta \in \Omega_{s,a}}{\sup}\mathbf{P}_{\beta}\left( \frac{|\tilde{\eta}^{m} - \eta|}{s} \geq \omega_p \right) = 0,
$$
for some $\omega_p \to 0$.
\end{thm}
It comes out that $\tilde{\eta}^{m}$  achieves almost full recovery under the nearly optimal condition
$$a \geq (1+\epsilon)\frac{\sigma\sqrt{2\log{(ep/s)}}}{\|X\|_{2,\infty}},$$ for any $ \epsilon \geq \delta$ as long as $s/p \to 0$. This sufficient condition is moreover optimal as $\delta \to 0$ by reduction to the Gaussian sequence model studied in \cite{butucea2018}. These results improve upon state-of-the-art recovery results in compressed sensing, and in particular results of \cite{ndaoud2018}, where authors use a two-stage procedure and sample splitting leading them to a strict loss in the sharp constants. Again our results, as opposed to most of the literature of support recovery, do not assume the design to be Gaussian nor sub-Gaussian. 
\begin{rque}
All sharp results are stated for given $\epsilon$. The procedure we construct depends on $\epsilon$. This is due to the fact that we do not have access to a sharp upper bound on $\delta$. In compressed sensing under isotropic sub-Gaussian design, we can replace $\epsilon$ by $\frac{s\log(ep/s)}{n}$ for instance. In this case, we can achieve sharp optimal results for any level $\epsilon$.
\end{rque}
In order to prove similar results for support recovery we rely on a different proof strategy. Our result for support recovery does not require sample splitting compared to \cite{ndaoud2018} and is more general than \cite{gao2019iterative} since it is fully adaptive. 
For any $\epsilon>0$, the statistical threshold for support recovery is given by
\begin{equation}\label{eq:def:threshold:3}
    \mu_{\infty}^{\epsilon} = (1+\sqrt{\epsilon}) \frac{\sigma\sqrt{2\log(p)}}{\|X\|_{2,\infty}},
\end{equation}
Define the least square solution, given the true support $S$, such that
\[
\tilde{\beta}^{*} = ((X^\top X)_{SS})^{-1}X_{S}^\top Y.
\]
Then the following result holds.
\begin{thm}\label{thm:sharp:4}
Assume that conditions of Theorem \ref{thm:sharp:1} hold, and $a \geq (1+3\sqrt{\epsilon})\frac{\sigma(\sqrt{2\log{(p)}}+\sqrt{2\log(s)})}{\|X\|_{2,\infty}}.$ Let $(\tilde{\beta}^{m})_m$ be the sequence of estimators defined in \eqref{eq:est:sharp} corresponding to $\lambda=\mu_{\infty}^{\epsilon}$ defined in \eqref{eq:def:threshold:3}. Then, we have for all $m \geq 0$ that
$$
 \underset{s/p \to 0}{\lim}\underset{\beta \in \Omega_{s,a}}{\sup}\mathbf{P}_{\beta}\left( \| \tilde{\beta}^m - \tilde{\beta}^{*}\|^2 \geq \frac{150^2(10\delta)^{m}\sigma^2s\log(ep/s)}{\|X\|_{2,\infty}^{2}} \right) = 0,
$$
 As a consequence, we get that for $m \geq \log(s)$
$$
\underset{s/p \to 0}{\lim} \underset{\beta \in \Omega_{s,a}}{\sup}\mathbf{P}_{\beta}\left( |\tilde{\eta}^{m} - \eta| > 0 \right) = 0.
$$
\end{thm}
Hence under the minimal separation condition for support recovery our estimator converges to the oracle least square solution $\tilde{\beta}^m \to \tilde{\beta}^{*}$ in probability even for non vanishing $\delta$. Notice that $\tilde{\beta}^{*}$ has the same support as $\beta$ a.s.  It turns out that 
$$a \geq (1+3\sqrt{\epsilon})\frac{\sigma(\sqrt{2\log{(p)}}+\sqrt{2\log(s)})}{\|X\|_{2,\infty}},$$ for any $\epsilon>0$ is sufficient to achieve exact recovery as long as $\delta \to 0$, $s/p \to 0$. This condition is shown to be necessary in \cite{butucea2018} under orthogonal design.
We also recover the results of \cite{gao2019iterative} for Gaussian design. Moreover, our approach is fully adaptive and holds beyond Gaussian design.
\section{Conclusion}

In this paper, we have presented a novel non-asymptotic minimax optimal estimation procedure for high dimensional linear regression. Our procedure is moreover fast and fully adaptive. We also provided sharp asymptotic results beyond the Gaussian design assumption. In particular, our procedure is scaled minimax optimal (i.e. unbiased whenever it is possible). Moreover, optimal results for both exact and almost full recovery were established as $\delta \to 0$. We conclude that our procedure has many attractive properties under the high dimensional linear regression model. 

As potential extensions of our results, we believe that full adaptation with respect to the optimization parameters $\delta_s$ and $L_{s}$ has its own interest. Moreover, our results in their actual form do not have the optimal dependence in terms of the condition number $\gamma_s$, and it would be interesting to generalize them beyond the RIP condition. Finally, another direction of interest is robust estimation through algorithmic regularization, where our strategy may be used to construct robust estimators with some of the desired properties we have in this paper. We leave all these questions for further research.
\section*{Acknowledgements}

 I would like to thank Alexandre Tsybakov and Pierre Bellec for valuable comments on early versions of this manuscript. This work was partially supported by a James H. Zumberge Faculty Research and Innovation Fund at the University of Southern California and by the National Science Foundation grant CCF-1908905.



\bibliographystyle{imsart-nameyear}

\newpage
\appendix

\section{Proofs of non-asymptotic results}
\begin{proof}[\textbf{Proof of Lemma \ref{lem:contraction}}]
Recall that $X$ satisfies RIP($3s,\delta/2$). It is easy to observe that $m_{3s} \leq \|X\|^2_{2,\infty} \leq L_{3s}$.
Hence for $S \in \{1,\dots,p\}$ such that $|S|\leq 3s$, we have
$$
\lambda_{\max}(\Phi_{SS}) \leq \left(1-\frac{m_{3s}}{L_{3s}} \right) \vee \left( \frac{L_{3s}}{m_{3s}}-1 \right).
$$
Since $1-\frac{m_{3s}}{L_{3s}} \leq \delta/2$, it remains to prove that
$$
\frac{L_{3s}}{m_{3s}}-1 \leq \delta.
$$
Using that fact that $m_{3s}\geq(1-\delta/2)L_{3s}$, it comes that
$$
\frac{L_{3s}}{m_{3s}}-1 \leq \frac{\delta/2}{1-\delta/2}\leq \delta,
$$
since $\delta\leq1$.
\end{proof}
\begin{proof}[\textbf{Proof of Theorem \ref{thm:main}}]
We proceed by induction. For $m=0$, The result is obvious. We now assume the result true for $m$ and prove it for $m+1$. In what follows let us denote by $H^{m+1}$ the vector
$$
H^{m+1} = \hat{\beta}^{m} + \frac{1}{\|X\|_{2,\infty}^2}X^{\top}(Y-X\hat{\beta}^{m}).
$$
Notice that $H^{m+1}$ can be written in the form
$$
H^{m+1} = \beta + \Phi(\beta - \hat{\beta}^{m}) + \Xi,
$$
and that 
$$
\hat{\beta}^{m+1} = \mathbf{T}_{\lambda_{m+1}}(H^{m+1}).
$$
We prove the first part of the result reasoning by the absurd. Assume that $|\hat{\beta}^{m+1}_{S^{c}}|_{0} > s$. Then there exists a subset $\tilde{S}$ of $S^{c}$ such that $|\tilde{S}|_{0}=s$ and 
$$
s\lambda_{m+1}^{2} \leq \sum_{i \in \tilde{S}}(H^{m+1}_{i})^{2}\mathbf{1}\{|H^{m+1}_{i}|\geq \lambda_{m+1}\}.
$$
Since $\tilde{S}$ is not supported on $S$, then we have
$$
\sqrt{s}\lambda_{m+1} \leq \sqrt{\sum_{i \in \tilde{S}}\Xi_{i}^{2}} + \sqrt{\sum_{i \in \tilde{S}} \langle \Phi_{i}^{\top}, \beta - \hat{\beta}^{m} \rangle^{2}}.
$$
Since $\beta$ is $s$-sparse and $|\hat{\beta}_{S^{c}}^{m}|_{0}\leq s$ then $\beta - \hat{\beta}^{m}$ is at most $2s$-sparse. Moreover $|\tilde{S}|=s$. Hence using Lemma \ref{lem:contraction} we have that 
$$
\sqrt{s}\lambda_{m+1} \leq \sqrt{\sum_{i=1}^{s}\Xi_{(i)}^{2}} + \delta\|\beta - \hat{\beta}^{m}\|.
$$
Using the induction hypothesis and event $\mathcal{O}$, we get moreover that
\begin{align*}
\sqrt{s}\lambda_{m+1} &\leq \frac{\sqrt{10\sigma^2s\log(ep/s)}}{\|X\|_{2,\infty}} +  3\delta\sqrt{s} \lambda_{m}\\
&\leq (1/2 + 3 \sqrt{\delta})\sqrt{s}\lambda_{m+1} < \sqrt{s}\lambda_{m+1},
\end{align*}
as long as $\delta < 1/36$, which is absurd. Hence $|\hat{\beta}_{S^{c}}^{m+1}|_{0} \leq s$. For the second part, observe that $\forall i \in S,$ 
\begin{align*}
 \hat{\beta}_{i}^{m+1}- \beta_{i} = -H_{i}^{m+1}\mathbf{1}\{|H^{m+1}_{i}| \leq \lambda_{m+1}\} 
 + \Xi_{i} + \langle \Phi_{i}^{\top},\beta - \hat{\beta}^{m} \rangle.
\end{align*}
Then using the same arguments as before we have 
$$
\| \hat{\beta}_{S}^{m+1} - \beta\| \leq \sqrt{s}\lambda_{m+1} + \frac{\sqrt{10\sigma^2s\log(ep/s)}}{\|X\|_{2,\infty}} +  \delta\|\hat{\beta}^{m} - \beta\|.
$$
Moreover on $S^{c}$, we have
$$
\| \hat{\beta}_{S^{c}}^{m+1} \| \leq \frac{\sqrt{10\sigma^2s\log(ep/s)}}{\|X\|_{2,\infty}} + \delta\|\hat{\beta}^{m} - \beta\|.
$$
Hence
$$
\| \hat{\beta}^{m+1} - \beta\| \leq \sqrt{s}\lambda_{m+1} + 2\frac{\sqrt{10\sigma^2s\log(ep/s)}}{\|X\|_{2,\infty}} + 2 \delta\|\hat{\beta}^{m} - \beta\|.
$$
Using the definition of $\lambda_m$ and the induction hypothesis, we get that
$$
\| \hat{\beta}^{m+1} - \beta\| \leq \sqrt{s}\lambda_{m+1} + \sqrt{s}\lambda_{m+1} + 6\delta \sqrt{s}\lambda_{m}.
$$
We conclude that
$$
\| \hat{\beta}^{m+1} - \beta\| \leq  \sqrt{s}\lambda_{m+1}(2 + 6 \sqrt{\delta}) \leq 3\sqrt{s}\lambda_{m+1}.
$$
\end{proof}
\begin{proof}[\textbf{Proof of Proposition \ref{prop:threshold:initial}}]
Let $\tilde{S}$ be a set of size $s$ then
$$
\|M_{\tilde{S}}\| \leq \|\beta_{\tilde{S}}\| + \delta \|\beta\|  + \frac{\sigma\sqrt{10s\log(ep/s)}}{\|X\|_{2,\infty}}.
$$
Hence 
$$
\|M_{\tilde{S}}\|  \leq   \|\beta\|(1+\delta) + \frac{\sigma\sqrt{10s\log(ep/s)}}{\|X\|_{2,\infty}} .
$$
Then
$$
\sqrt{s} \hat{\lambda}_0 \leq  2(\sqrt{10}(1+\delta) + 1)\left(\|\beta\|\vee \frac{\sigma\sqrt{10s\log(ep/s)}}{\|X\|_{2,\infty}}\right).
$$
Hence
$$
\sqrt{s} \hat{\lambda}_0 \leq  10\left(\|\beta\|\vee \frac{\sigma\sqrt{10s\log(ep/s)}}{\|X\|_{2,\infty}}\right).
$$
We also have for the true support $S$ of $\beta$ that
$$
\|M_{S}\| \geq (1-\delta)\|\beta\|-\frac{\sigma\sqrt{10s\log(ep/s)}}{\|X\|_{2,\infty}}.
$$
    If $\|\beta\| \leq \frac{\sigma\sqrt{40s\log(ep/s)}}{\|X\|_{2,\infty}}$ the result is trivial. Else $\|\beta\| > \frac{\sigma\sqrt{40s\log(ep/s)}}{\|X\|_{2,\infty}}$, and
$$
\sqrt{s}\hat{\lambda}_0 \geq \sqrt{10}(1-\delta - 1/2)\|\beta\| \geq \|\beta\|.
$$
The result for $\hat{m}$ is straightforward.
\end{proof}
\begin{proof}[\textbf{Proof of Theorem \ref{thm:optimal:statistic}}]
Assume that event $\mathcal{O}$ holds, then using Proposition \ref{prop:threshold:initial}, we have $\|\beta\|^{2} \leq s\hat{\lambda}_0^{2}$. We can then apply Theorem \ref{thm:main} and get that
$$
\forall m \geq 0, \quad \|\hat{\beta}^{m} - \beta\|^2 \leq 9\left(s\hat{\lambda}_0^{2}\kappa^{m/2} \vee \frac{40\sigma^2s\log(ep/s)}{\|X\|_{2,\infty}^2}\right).
$$
With the choice of $\hat{m}$ we get further using  Proposition \ref{prop:threshold:initial} that
$$
s\hat{\lambda}_0^{2}\kappa^{\hat{m}/2} \leq \frac{40\sigma^2s\log(ep/s)}{\|X\|_{2,\infty}^2}.
$$
Hence
$$
\|\hat{\beta}^{\hat{m}} - \beta\|^{2} \leq 360 \frac{\sigma^{2}s\log(ep/s)}{\|X\|_{2,\infty}^2}.
$$
It follows that
\begin{align*}
\mathbf{P}&\left( \| \hat{\beta}^{\hat{m}} - \beta \|^{2} \geq 
 360\frac{\sigma^{2}}{\|X\|_{2,\infty}^2}s\log(ep/s) \right) \leq \\
 & \mathbf{P}\left( \sum_{i=1}^{s}\Xi_{(i)}^{2}\geq \frac{10\sigma^{2}s\log(ep/s)}{\|X\|_{2,\infty}^2} \right).
\end{align*}
We conclude using Lemma \ref{lem:order:statistic}. We proceed similarly for the remaining  statements using Proposition \ref{prop:threshold:initial}.
\end{proof}
\begin{proof}[\textbf{Proof of Theorem \ref{thm:optimal:statistic:adap}}]
For this proof we consider both events $\mathcal{O}$ and $\mathcal{A}$ where
$$
\mathcal{A}=\{ |\|\xi\| - \sqrt{n}| \leq 1/4 \sqrt{n}\}.
$$
Using the Hanson-Wright inequality \cite{rudelson2013hanson} and the condition on $n$, it is easy to observe that
$$
\mathbf{P}(\mathcal{A}) \leq e^{-c_2s\log(ep/s)},
$$
for some absolute  $c_2>0$.
For the rest of the proof we place ourselves on the event $\mathcal{O}\cap \mathcal{A}$ that holds with probability $1-e^{-c_3s\log(ep/s)}$ for some absolute $c_3 >0$.
From the definition of $\hat{\sigma}_m^2$, observe that, as long as $\beta - \hat{\beta}_m$ is $2s$- sparse, we have
$$
|\hat{\sigma}_m - \sigma| \leq \frac{\|X\|_{2,\infty}}{\sqrt{n}}(1 + \delta)\|\beta - \hat{\beta}_m\| + \frac{1}{4}\sigma.
$$
Hence
\begin{equation}\label{eq:bound:sigma:hat}
\frac{\hat{\sigma}_m}{\|X\|_{2,\infty}}\sqrt{160\log(ep)} \leq \frac{3\sqrt{40\log(ep)}}{\sqrt{n}}\|\beta - \hat{\beta}_m\| + \frac{3\sigma}{\|X\|_{2,\infty}}\sqrt{40\log(ep)}.
\end{equation}
Next, recall that
$$
    \bar{\lambda}_0 =\sqrt{20}|M|_{(1)} \vee\frac{\hat{\sigma}_0}{\|X\|_{2,\infty}} \sqrt{160\log(ep)}.
    $$
Since $n$ is large enough compared to $s\log(ep)$ it is easy to see that 
$$
    \bar{\lambda}_0  \leq \sqrt{20}|M|_{(s)}\vee \|\beta\| \vee \frac{12\sigma}{\|X\|_{2,\infty}}\sqrt{10 \log(ep)}.
    $$
Hence using Proposition \ref{prop:threshold:initial} it comes that
\begin{equation}\label{eq:bew:thres}
 \sqrt{2}\|\beta\| \leq \sqrt{s}\bar{\lambda}_0  \leq 20\sqrt{s}\left(\|\beta\|\vee \frac{\sigma\sqrt{10s\log(ep)}}{\|X\|_{2,\infty}}\right).
\end{equation}
Hence the threshold $\bar{\lambda}_0$ satisfies the condition required to apply Theorem \ref{thm:optimal:statistic}. Let $\hat{m}^*$ such that
\begin{equation}\label{eq:stop:adap:fictif}
    \hat{m}^* = \inf\left\{ m / \lambda_m \leq \frac{6\sigma}{\|X\|_{2,\infty}}\sqrt{40\log(ep)} \right\} + 1.
\end{equation}
We recall that
\begin{equation}
    \hat{m} = \inf\left\{ m / \lambda_m \leq \frac{\sigma}{\|X\|_{2,\infty}}\sqrt{40\log(ep/s)} \right\} + 1.
\end{equation}
Observe that $\hat{m}^* \leq \hat{m}$. As long as $m \leq \hat{m}^*$ then 
$$
\frac{\sigma}{\|X\|_{2,\infty}}\sqrt{40\log{(ep)}} \leq 1/6\lambda_m,
$$
so the first induction steps remain the same as before. Now using \eqref{eq:bound:sigma:hat} we get further
$$
\frac{\hat{\sigma}_m}{\|X\|_{2,\infty}}\sqrt{160\log{(ep)}} \leq \sqrt{\frac{3500s\log(ep)}{n}}\lambda_m + 1/2 \lambda_m.
$$
Hence for $n$ larger than  $14000s\log(ep)$, $\frac{\hat{\sigma}_m}{\|X\|_{2,\infty}}\sqrt{160\log{(ep)}}$ is strictly smaller than $\lambda_m$ and $\hat{m}^* < \bar{m}$.
After running $\hat{m}^*$ steps we get
$$
\|\beta - \hat{\beta}_{\hat{m}^*}\| \leq \frac{18\sigma}{\|X\|_{2,\infty}}\sqrt{40s\log{(ep)}},
$$
and for all $\hat{m}^* \leq m \leq \hat{m}$
$$
|\hat{\sigma}_{m} - \sigma| \leq \sigma/2.
$$
It comes out that for all $\hat{m}^* \leq m \leq \hat{m}$
$$
\frac{\hat{\sigma}_{m}}{\|X\|_{2,\infty}}\sqrt{160\log{(ep)}} \geq \frac{\sigma}{\|X\|_{2,\infty}}\sqrt{40\log{(ep)}}.
$$
Hence a finite number of steps after the first $\hat{m}^*$  are enough to hit the threshold $\frac{\hat{\sigma}_{\bar{m}}}{\|X\|_{2,\infty}}\sqrt{160\log{(ep)}}$ and stop the algorithm. Observe that $\bar{m} \leq \hat{m}$ and hence $\frac{\hat{\sigma}_{\bar{m}}}{\|X\|_{2,\infty}}\sqrt{160\log{(ep)}} \geq \hat{\lambda}_{\infty}$.
We finally get, applying Theorem \ref{thm:main}, that
$$
\|\beta - \hat{\beta}_{\bar{m}}\| \leq \frac{10\sigma}{\|X\|_{2,\infty}}\sqrt{40s\log{(ep)}},
$$
and that $|\hat{\beta}_{\bar{m}}|_{0} \leq 2s$.
Going back to the definition of $\bar{m}$ and using \eqref{eq:bew:thres} we get also that
$$
\bar{m} \leq 2\left(\log\left( \frac{10\|\beta\|^2\|X\|^2_{2,\infty}}{\sigma^2\log(ep)} \vee 100\right)/\log{(1/\kappa)}\right).
$$
This concludes the proof.
\end{proof}
\begin{proof}[\textbf{Proof of Theorem \ref{thm:optimal:statistic:adap2}}]
Before proving the result, we recall that with probability $1-e^{-cs\log(ep/s)}$ we have
$$
|\hat{\sigma}_{\bar{m}} - \sigma| \leq \frac{\|X\|_{2,\infty}}{\sqrt{n}}(1 + \delta)\|\beta - \hat{\beta}_{\bar{m}}\| + \frac{1}{20}\sigma \leq \sigma\left((1+\delta)\frac{10\sqrt{40}}{14000} +\frac{1}{20}\right) \leq \sigma/10.
$$
In what follows we assume that
\begin{equation}\label{eq:var}
    |\hat{\sigma}_{\bar{m}} - \sigma| \leq \sigma /10.
\end{equation} 
Using Theorem \ref{thm:optimal:statistic} and Lemma \ref{lem:barber}, then we  assume moreover that 
$$
 \| \hat{\beta}^{\hat{m}} - \beta \|^{2} \leq \frac{360\sigma^{2}s\log(ep/s) }{\|X\|_{2,\infty}^{2}},
$$
$$
 |\hat{\beta}_{S^c}^{\hat{m}}|_0\leq s,
$$
and 
\[
 \left\langle \xi, \frac{X^{\top}(\beta - \hat{\beta})}{\|X(\beta-\hat{\beta})\|}\right\rangle^2 \leq 7(s + |\hat{\beta}_{S^c}|_0)\log(ep/(s+|\hat{\beta}_{S^c}|_0)),
\]
since all those events hold with probability $1-e^{-cs\log(ep/s)}$. The remainder of the proof is fully deterministic. Based on \eqref{eq:var} it is easy to observe that $\hat{m} \leq \hat{T}$. Hence we have that
\begin{equation}\label{eq:main:111}
    \frac{1}{n}\|Y - X\hat{\beta}^{\tilde{m}} \|^2 + \frac{1000\hat{\sigma}_{\bar{m}}^2 |\hat{\beta}^{\tilde{m}}|_0\log(ep/|\hat{\beta}^{\tilde{m}}|_0)}{n} \leq \frac{1}{n}\|Y - X\hat{\beta}^{\hat{m}} \|^2 + \frac{1000\hat{\sigma}_{\bar{m}}^2 |\hat{\beta}^{\hat{m}}|_0\log(ep/|\hat{\beta}^{\hat{m}}|_0)}{n}.
\end{equation}
The rest of the proof is decomposed in two parts. 
\begin{itemize}
    \item \textit{Show that $|\hat{\beta}^{\tilde{m}}|_0 \leq 3s$:}
    
Let us assume that $|\hat{\beta}^{\tilde{m}}|_0 > 3s$. On the one hand, we have
\begin{align*}
    \|Y - X\hat{\beta}^{\tilde{m}} \|^2 &\geq \sigma^2\|\xi\|^2 + \| X(\beta-\hat{\beta}^{\tilde{m}})\|^2 - 2\sigma\left|\left\langle \xi,X(\beta - \hat{\beta}^{\tilde{m}}  )\right\rangle \right|\\
    &\geq \sigma^2\|\xi\|^2 + \| X(\beta-\hat{\beta}^{\tilde{m}})\|^2 - \sqrt{42\sigma^2|\hat{\beta}^{\tilde{m}}|_0\log(3ep/4|\hat{\beta}^{\tilde{m}}|_0)}\| X(\beta-\hat{\beta}^{\tilde{m}})\|\\
    &\geq \sigma^2\|\xi\|^2+ \| X(\beta-\hat{\beta}^{\tilde{m}})\|^2/2 - 21\sigma^2|\hat{\beta}^{\tilde{m}}|_0\log(ep/|\hat{\beta}^{\tilde{m}}|_0).
\end{align*}
It comes out that
\[
\frac{1}{n}\|Y - X\hat{\beta}^{\tilde{m}} \|^2 + \frac{1000\hat{\sigma}_{\bar{m}}^2 |\hat{\beta}^{\tilde{m}}|_0\log(ep/|\hat{\beta}^{\tilde{m}}|_0)}{n} \geq \frac{\sigma^2\|\xi\|^2}{n} + \frac{900\hat{\sigma}_{\bar{m}}^2 |\hat{\beta}^{\tilde{m}}|_0\log(ep/|\hat{\beta}^{\tilde{m}}|_0)}{n}.
\]
On the other hand, we have
\begin{align*}
    \|Y - X\hat{\beta}^{\bar{m}} \|^2 &\leq \sigma^2\|\xi\|^2 + \| X(\beta-\hat{\beta}^{\bar{m}})\|^2 + 2\sigma\left|\left\langle \xi,X(\beta - \hat{\beta}^{\bar{m}}  )\right\rangle \right|\\
    &\leq \sigma^2\|\xi\|^2 + \| X(\beta-\hat{\beta}^{\bar{m}})\|^2 + \sqrt{8\sigma s\log(ep/2s)}\| X(\beta-\hat{\beta}^{\bar{m}})\|\\
    &\leq \sigma^2\|\xi\|^2 + 3\| X(\beta-\hat{\beta}^{\tilde{m}})\|^2/2 + 4\sigma^2s\log(ep/2s).
\end{align*}
It comes out that
\[
\frac{1}{n}\|Y - X\hat{\beta}^{\bar{m}} \|^2 + \frac{1000\hat{\sigma}_{\bar{m}}^2 |\hat{\beta}^{\bar{m}}|_0\log(ep/|\hat{\beta}^{\bar{m}}|_0)}{n} \leq \frac{\sigma^2\|\xi\|^2}{n} + \frac{2500\hat{\sigma}_{\bar{m}}^2 s\log(ep/2s)}{n}.
\]
Going back to \eqref{eq:main:111}, we conclude that
\[
\frac{900 |\hat{\beta}^{\tilde{m}}|_0\log(ep/|\hat{\beta}^{\tilde{m}}|_0)}{n} \leq \frac{2500 s\log(ep/2s)}{n}.
\]
The last equation does not hold as long as $s/p$ is small enough. As a consequence, we have  $|\hat{\beta}^{\tilde{m}}|_{0} \leq 3s$.
    \item \textit{Show that $\|\hat{\beta}^{\tilde{m}} - \beta \| \leq 100\sigma\sqrt{s\log(ep/s)}/\|X\|_{2,\infty}$:}
    Using the above equations we get also that
    \[
    \| X(\beta-\hat{\beta}^{\tilde{m}})\|^2/(2n) \leq \frac{2500\hat{\sigma}_{\bar{m}}^2 s\log(ep/es)}{n}.
    \]
    Since $\hat{\beta}^{\tilde{m}}$ is at most $3s$-sparse, then we conclude using RIP that
    \[
    \| \beta-\hat{\beta}^{\tilde{m}}\|^2 \leq \frac{100^2\sigma^2 s\log(ep/s)}{\|X\|^{2}_{2,\infty}}.
    \]
    The result corresponding to $\hat{T}$ is straightforward based on \eqref{eq:var}.
\end{itemize}
\end{proof}

\section{Proofs of asymptotic results}
Without loss of generality we will assume in the next proofs that $\| \tilde{\beta}^{0} - \beta\|^2 \leq \frac{2.10^4\sigma^{2}s\log(ep/s) }{\|X\|_{2,\infty}^{2}} $ and $|\tilde{\beta}^{0}|_{S^{c}}\leq s$. Indeed according to Theorem \ref{thm:optimal:statistic:adap2} both statements hold with high probability. In order to alleviate notations we  assume that $\sigma=\|X\|_{2,\infty}$.
\begin{proof}[\textbf{Proof of Theorem \ref{thm:sharp:1}}]
Let $\epsilon>0$. 
Observe that 
\begin{align*}
\mathbf{P}&\left(\sum_{i=1}^{p}\Xi_{i}^{2}\mathbf{1}\{|\Xi_{i}|
\geq \lambda\} \geq \frac{s}{\log\log(ep/s)}\right)\\
&\leq \frac{\log\log (ep/s)}{s}\sum_{i=1}^{p}\mathbf{E}\left(\Xi_{i}^{2}\mathbf{1}\{|\Xi_{i}|\geq \lambda\} \right).
\end{align*}
Hence using Lemma \ref{lem:subG}, we get
$$
\mathbf{P}\left(\sum_{i=1}^{p}\Xi_{i}^{2}\mathbf{1}\{|\Xi_{i}|
\geq \lambda_{\infty}^{0}(1+\sqrt{\epsilon}/2)\} \geq \frac{s}{\log\log( ep/s)}\right) = o(1).
$$
Using Lemma \ref{lem:noise_control}, we get also that
$$
\mathbf{P}\left(\|\Xi_{S}\|^{2} \geq s\sqrt{\log(ep/s)} \right)= o(1).
$$
As a consequence, we assume, in what follows,  that $\sum_{i=1}^{p}\Xi_{i}^{2}\mathbf{1}\{|\Xi_{i}|\geq \lambda_{\infty}^{0}(1+\sqrt{\epsilon}/2)\}\leq \frac{s}{\log\log(ep/s)} $ and $\|\Xi_{S}\|^{2} \leq s\sqrt{\log(ep/s)}$.
We show first that $\forall m \geq 0$ we have 
$$
 |\tilde{\beta}^{m}_{S^{c}}|_{0} \leq s,
$$
and
$$
\quad \|\tilde{\beta}^{m} - \beta\| \leq \sqrt{s}(\lambda+ 2\log(ep/s)^{1/4})(1+4\sqrt{\delta} + 100 \delta^{m/2}).
$$
For $m=0$ the result is immediate based on the assumption on $\lambda$. Assume that result holds for $m$ and we prove it for $m+1$. On $S$ we have
$$
|\tilde{\beta}_{i}^{m+1} - \beta_{i}| \leq \lambda + |\Xi_{i}| + |\langle \Phi_{i}^{\top},\beta - \tilde{\beta}^{m}\rangle|.
$$
Hence
$$
\| \beta - \tilde{\beta}^{m+1}_{S} \| \leq \sqrt{s}(\lambda + \log(ep/s)^{1/4}) + \delta\|\beta-\tilde{\beta}^{m}\|.
$$
On $S^{c}$, we show first that $|\hat{\beta}^{m+1}_{S^{c}}|_{0}\leq s$. By absurd, assuming this is not the case, then
\begin{align*}
\sqrt{s} \lambda &\leq \sqrt{\sum_{i \in \tilde{S}} |\Xi_{i}|^{2}\mathbf{1}\{|\Xi_{i}|\geq \lambda_{\infty}^{\epsilon/4}\}} \\
&+ \sqrt{\sum_{i \in \tilde{S}} |\Xi_{i}|^{2}\mathbf{1}\{|\Xi_{i}|\leq \lambda^{\epsilon/4}_{\infty}, |\langle \Phi_{i},\beta - \tilde{\beta}^{m} \rangle| \leq \sqrt{\epsilon}\lambda_{\infty}^{0}/2\}}  + \delta\|\beta-\tilde{\beta}^{m}\|
\end{align*}
for some $\tilde{S}$ that is $s$-sparse. It comes out that
$$
\sqrt{s} \lambda \leq \sqrt{\frac{s}{\log\log(ep/s)}} + 2\left(1+\frac{1}{\sqrt{\epsilon}}\right)  \delta\|\beta-\tilde{\beta}^{m}\|.
$$
Using the induction hypothesis
\begin{align*}
\sqrt{s} \lambda &\leq \sqrt{\frac{s}{\log\log(ep/s)}} \\
&+ 2\sqrt{\delta}(\sqrt{\delta}+1)\sqrt{s}(\lambda + 2\log(ep/s)^{1/4})(1+4\sqrt{\delta} + 100 \delta^{m/2}).
\end{align*}
Since $\delta \leq 1/400^2$ and $p/s$ large enough the above statement can not hold. Next we have on $S^{c}$
$$
|\tilde{\beta}_{i}^{m+1}| \leq   |\Xi_{i}|\mathbf{1}\{|H^{m+1}_{i}|\geq \lambda\} + |\langle \Phi_{i}^{\top},\beta - \tilde{\beta}^{m}\rangle|.
$$
Since $S_{m+1}^{c}$ is $s$-sparse then arguing as above we get 
$$
\|\tilde{\beta}^{m+1}_{S^{c}} \| \leq  2\left(1+\frac{1}{\sqrt{\epsilon}}\right)\delta\|\beta-\tilde{\beta}^{m}\|+\sqrt{\frac{s}{\log\log(ep/s)}}.
$$
Hence
$$
\|\tilde{\beta}^{m+1} - \beta\| \leq \sqrt{s}(\lambda + 2\log(ep/s)^{1/4}) + \sqrt{\delta}(3+2\sqrt{\delta} )\|\beta-\tilde{\beta}^{m}\|.
$$
Using the induction hypothesis
\begin{align*}
&\|\tilde{\beta}^{m+1} - \beta\| \leq \\
&\sqrt{s}(\lambda + 2\log(ep/s)^{1/4})(1 + \sqrt{\delta}(3 + 2 \sqrt{\delta})(1+4\sqrt{\delta}+100\delta^{m/2})).
\end{align*}
Since $\sqrt{\delta}(3 + 2 \sqrt{\delta})(1+4\sqrt{\delta}+100\delta^{m/2}) \leq 4\sqrt{\delta} + 100 \delta^{(m+1)/2}$ then the result follows.
As a consequence, as $s/p \to 0, $we have for $m \geq 0$ 
$$
\|\tilde{\beta}^{m} - \beta\| \leq (1+4\sqrt{\delta}+100\delta^{m/2})\sqrt{s}\lambda.
$$
This concludes the proof.
\end{proof}
\begin{proof}[\textbf{Proof of Theorem \ref{thm:sharp:2}}]
Let $\epsilon>0$. Following the same steps as in the proof of Theorem \ref{thm:sharp:1}, we assume, in what follows,  that $\sum_{i=1}^{p}\Xi_{i}^{2}\mathbf{1}\{|\Xi_{i}|\geq \lambda_{\infty}^{0}(1+\sqrt{\epsilon}/2)\}\leq \frac{s}{\log\log(ep/s)} $, $\sum_{i \in S}\mathbf{1}\{|\Xi_{i}|\geq \lambda_{\infty}^{0}\sqrt{\epsilon}/2\}\leq \frac{s}{2\log(ep/s)\log\log(ep/s)}$ and $\|\Xi_{S}\|^{2} \leq s+ \log(s)$.
Indeed, sing Lemma \ref{lem:subG}, we get
$$
\mathbf{P}\left(\sum_{i=1}^{p}\Xi_{i}^{2}\mathbf{1}\{|\Xi_{i}|
\geq \lambda_{\infty}^{0}(1+\sqrt{\epsilon}/2)\} \geq \frac{s}{\log\log( ep/s)}\right) = o(1),
$$
and using Markov inequality
$$
\mathbf{P}\left(\sum_{i \in S}\mathbf{1}\{|\Xi_{i}|\geq \lambda_{\infty}^{0}\sqrt{\epsilon}/2\}\geq \frac{s}{2\log(ep/s)\log\log(ep/s)}\right)= o(1) .
$$
Using Lemma \ref{lem:noise_control}, we get also that
$$
\mathbf{P}\left(\|\Xi_{S}\|^{2} \geq s + \log(s) \right)= o(1),
$$
as $s \to \infty$. 

We show that $\forall m \geq 0$ we have 
$$
 |\tilde{\beta}^{m}_{S^{c}}|_{0} \leq s,
$$
and
$$
\quad \|\tilde{\beta}^{m} - \beta\| \leq 100 \sqrt{s}\lambda \delta^{m/2}+ \left(\sqrt{s+\log(s)} + \sqrt{\frac{4s}{\log\log(ep/s)}}\right)(1+4\sqrt{\delta}).
$$
For $m=0$ the result is immediate based on the assumption on $\lambda$. Assume the result is true for $m$ and we prove it for $m+1$. On $S$ we have
\begin{align*}
|\tilde{\beta}_{i}^{m+1} - \beta_{i}| & \leq \lambda\mathbf{1}\{|\beta_{i} +\Xi_{i}+\langle \Phi_{i}^{\top},\beta - \tilde{\beta}^{m} \rangle| \leq \lambda\} \\
&+| \Xi_{i}| + |\langle \Phi_{i}^{\top},\beta - \tilde{\beta}^{m} \rangle|.
\end{align*}
Moreover
\begin{align*}
\mathbf{1}\{|\beta_{i}  +\Xi_{i}+\langle \Phi_{i}^{\top},\beta - \tilde{\beta}^{m} \rangle| \leq \lambda\}  &\leq \mathbf{1}\{|\Xi_{i}|+|\langle \Phi_{i}^{\top},\beta - \hat{\beta}^{m}| \geq (|\beta_{i}|-\lambda)\}\\
& \leq \mathbf{1}\{|\Xi_{i}|\geq \sqrt{\epsilon}\lambda/2\} +  \mathbf{1}\{|\langle \Phi_{i}^{\top},\beta - \hat{\beta}^{m}\rangle| \geq \sqrt{\epsilon}\lambda/2\}.
\end{align*}
Hence
$$
\| \beta - \tilde{\beta}^{m+1}_{S} \| \leq \sqrt{\frac{s}{\log\log(ep/s)}} + \sqrt{s+\log(s)}+ \delta\|\beta-\tilde{\beta}^{m}\|\left(1+\frac{2}{\sqrt{\epsilon}}\right).
$$
On $S^{c}$, we show first that $|\tilde{\beta}^{m+1}_{S^{c}}|_{0}\leq s$. By absurd assume this is not the case then
\begin{align*}
\sqrt{s} \lambda &\leq \sqrt{\sum_{i \in \tilde{S}} |\Xi_{i}|^{2}\mathbf{1}\{|\Xi_{i}|\geq \lambda_{\infty}^{\epsilon/4}\}} \\
&+ \sqrt{\sum_{i \in \tilde{S}} |\Xi_{i}|^{2}\mathbf{1}\{|\Xi_{i}|\leq \lambda^{\epsilon/4}_{\infty}, |\langle \Phi_{i},\beta - \tilde{\beta}^{m} \rangle| \leq \sqrt{\epsilon}\lambda_{\infty}^{0}/2\}}  + \delta\|\beta-\tilde{\beta}^{m}\|
\end{align*}
for some $\tilde{S}$ that is $s$-sparse. It comes out that
$$
\sqrt{s} \lambda \leq \sqrt{\frac{s}{\log\log(ep/s)}} + 2\left(1+\frac{1}{\sqrt{\epsilon}}\right)  \delta\|\beta-\tilde{\beta}^{m}\|.
$$
Using the induction hypothesis
\begin{align*}
\sqrt{s} \lambda \leq \sqrt{\frac{s}{\log\log(ep/s)}} + 2\sqrt{\delta}(\sqrt{\delta}+1)\|\beta-\tilde{\beta}^{m}\|.
\end{align*}
Since $\delta \leq 1/400^2$ and both $s$ and $p/s$ are large enough the above statement can not hold. Next we have on $S^{c}$
$$
|\tilde{\beta}_{i}^{m+1}| \leq   |\Xi_{i}|\mathbf{1}\{|H^{m+1}_{i}|\geq \lambda\} + |\langle \Phi_{i}^{\top},\beta - \tilde{\beta}^{m}\rangle|.
$$
Since $S_{m+1}^{c}$ is $s$-sparse then arguing as above we get 
$$
\|\tilde{\beta}^{m+1}_{S^{c}} \| \leq  2\left(1+\frac{1}{\sqrt{\epsilon}}\right)\delta\|\beta-\tilde{\beta}^{m}\|+\sqrt{\frac{s}{\log\log(ep/s)}}.
$$
Hence
$$
\|\tilde{\beta}^{m+1} - \beta\| \leq \sqrt{\frac{4s}{\log\log(ep/s)}} + \sqrt{s+\log(s)}+ \sqrt{\delta}(3+2\sqrt{\delta} )\|\beta-\tilde{\beta}^{m}\|.
$$
We conclude using the induction hypothesis
since $\sqrt{\delta}(3 + 2 \sqrt{\delta})(1+4\sqrt{\delta}+100\delta^{m/2}) \leq 4\sqrt{\delta} + 100 \delta^{(m+1)/2}$.
As a consequence, as $s/p \to 0$ and $s\to \infty$, we have for $m \geq 0$ that 
$$
\|\tilde{\beta}^{m} - \beta\| \leq 100 \sqrt{s}\lambda \delta^{m/2}+ \sqrt{s}(1+4\sqrt{\delta}+o(1)).
$$
Observing that for $m \geq \log\log(p/s)$ we have $100 \sqrt{s}\lambda \delta^{m/2} = o (\sqrt{s})$ concludes the proof.
\end{proof}

\begin{proof}[\textbf{Proof of Theorem \ref{thm:sharp:3}}]
Similarly to the proof of Theorem \ref{thm:sharp:2} we assume that $\sum_{i=1}^{p}\Xi_{i}^{2}\mathbf{1}\{|\Xi_{i}|\geq \lambda_{\infty}^{0}(1+\sqrt{\epsilon}/2)\}\leq \frac{s}{\log\log(ep/s)} $, $\sum_{i \in S}\mathbf{1}\{|\Xi_{i}|\geq \lambda_{\infty}^{0}\sqrt{\epsilon}/2\}\leq \frac{s}{2\log(ep/s)\log\log(ep/s)}$ and $\|\Xi_{S}\|^{2} \leq s+ \log(s) + \log(p/s)$. Observe that $\|\Xi_{S}\|^{2} \leq s+ \log(s) + \log(p/s)$ holds with high probability even for fixed $s$. It comes out that
\begin{align*}
\sum_{i \in S}|\tilde{\eta}^m_{i} -\eta_{i}| &= \sum_{i \in S} \mathbf{1}\{|H^m_{i}| \leq \lambda\} \\
&\leq \sum_{i \in S} \mathbf{1}\{ |\Xi_{i}| \geq \lambda\sqrt{\epsilon}/2 \} + \sum_{i \in S} \mathbf{1}\{ |\Phi_{i}^{\top},\beta-\tilde{\beta}^{m}| \geq \lambda\sqrt{\epsilon}/2\}.
\end{align*}
Hence using Theorem \ref{thm:sharp:2} we get, for $m \geq \log\log(ep/s)$, that
$$
\sum_{i \in S}|\tilde{\eta}^m_{i} -\eta_{i}| = o(s).
$$
Moreover on $S^{c}$ we have that $|\tilde{\eta}^m|_{0} \leq s$ and 
\begin{align*}
\sum_{i \in S^c}|\tilde{\eta}^m_{i} -\eta_{i}| &= \sum_{i  \in S^c} \mathbf{1}\{|H^m_{i}| \geq \lambda^{\epsilon}_{\infty}\} \\
& \leq \sum_{i   \in S^c} \mathbf{1}\{ |\Xi_{i}| \geq  \lambda_{\infty}(1+\sqrt{\epsilon}/2) \}\\
&+ \sum_{i  \in S^c} \mathbf{1}\{ |\Phi_{i}^{\top},\beta-\tilde{\beta}^{m}| \geq \lambda_{\infty}\sqrt{\epsilon}/2\}.
\end{align*}
Hence using again Theorem \ref{thm:sharp:2} we get, for $m \geq \log\log(ep/s)$, that
$$
\sum_{i \in S^c}|\tilde{\eta}^m_{i} -\eta_{i}| = o(s).
$$
The result follows immediatly.
\end{proof}

\begin{proof}[\textbf{Proof of Theorem \ref{thm:sharp:4}}]
So far we have used the following decomposition
$$
H^{m+1} = \beta + \Phi(\beta - \hat{\beta}^{m}) + \Xi.
$$
Using the oracle vector $\tilde{\beta}^*$, we get another decomposition
$$
H^{m+1} = \tilde{\beta^{*}} + \Phi(\tilde{\beta}^{*} - \hat{\beta}^{m}) + \tilde{\Xi},
$$
where $\tilde{\Xi}_{S} = 0$ and $\tilde{\Xi}_{S^c} = \frac{1}{\|X\|_{2,\infty}}X^{\top}_{S^c}\left(\mathbf{I}_n - X_{S}((X^\top X)_{SS})^{-1}X^\top_S\right)\xi$. Our goal is to estimate $\tilde{\beta}^{*}$. Observe that the noise $\tilde{\Xi}$ is zero on $S$ and each of its coordinates is $1$-subGaussian on $S^{c}$. Without loss of generality we may assume that $\|\tilde{\beta}^0 - \tilde{\beta}^{*}\| \leq 150 \sqrt{s\log(ep/s)}$ since $\|\beta - \tilde{\beta}^{*}\| \leq 50 \sqrt{s\log(ep/s)}$ with high probability.

Let $\epsilon>0$. We assume, in what follows,  that $\forall i =1,\dots,p, |\tilde{\Xi}_i| \leq \mu_{\infty}^{\epsilon/4}$.
This actually holds with high probability since each coordinate is $1$-subGaussian. Similarly, we also assume that $\|\tilde{\beta}^* - \beta \|_{\infty} \leq \sqrt{2\log(s)} +\sqrt{2\epsilon\log(p)}$. In particular, for all $i\in S$, $|\tilde{\beta}^*_i| \geq (1+2\sqrt{\epsilon})\sqrt{2\log(p)}$.

We show that $\forall m \geq 0$ we have 
$$
 |\tilde{\beta}^{m}_{S^{c}}|_{0} \leq s,
$$
and
$$
\quad \|\tilde{\beta}^{m} - \tilde{\beta}^*\| \leq 150 \sqrt{s\log(ep/s)} (10\delta)^{m/2}.
$$
For $m=0$ the result is immediate. Assume the result is true for $m$ and we prove it for $m+1$. On $S$ we have
\begin{align*}
|\tilde{\beta}_{i}^{m+1} - \tilde{\beta}^*_{i}|  \leq \mu_{\infty}^{\epsilon}\mathbf{1}\{|\tilde{\beta}^*_{i} +\langle \Phi_{i}^{\top},\tilde{\beta}^* - \tilde{\beta}^{m} \rangle| \leq \mu_{\infty}^{\epsilon}\}  + |\langle \Phi_{i}^{\top},\tilde{\beta}^* - \tilde{\beta}^{m} \rangle|.
\end{align*}
Moreover
\begin{align*}
\mathbf{1}\{|\tilde{\beta}^*_{i} +\langle \Phi_{i}^{\top},\tilde{\beta}^{*} - \tilde{\beta}^{m} \rangle| \leq \mu_{\infty}^\epsilon\}  &\leq \mathbf{1}\{|\langle \Phi_{i}^{\top},\tilde{\beta}^* - \hat{\beta}^{m}| \geq (|\tilde{\beta}^*_{i}|-\mu_{\infty}^\epsilon)\}\\
&  \leq \mathbf{1}\{|\langle \Phi_{i}^{\top},\tilde{\beta}^* - \hat{\beta}^{m}\rangle| \geq \sqrt{\epsilon}\mu_{\infty}^0\}.
\end{align*}
Hence
$$
\| \tilde{\beta}^* - \tilde{\beta}^{m+1}_{S} \| \leq  \delta\|\tilde{\beta}^*-\tilde{\beta}^{m}\|\left(2+\frac{1}{\sqrt{\epsilon}}\right).
$$
On $S^{c}$, we show first that $|\tilde{\beta}^{m+1}_{S^{c}}|_{0}\leq s$. By absurd assume this is not the case then
\begin{align*}
\sqrt{s} \mu_{\infty}^\epsilon \leq &\sqrt{\sum_{i \in \tilde{S}} |\tilde\Xi_{i}|^{2}\mathbf{1}\{|\tilde\Xi_{i}|\geq \mu_{\infty}^{\epsilon/4}\}} \\
&+ \sqrt{\sum_{i \in \tilde{S}} |\tilde{\Xi}_{i}|^{2}\mathbf{1}\{|\tilde{\Xi}_{i}|\leq \mu^{\epsilon/4}_{\infty}, |\langle \Phi_{i},\tilde{\beta}^* - \tilde{\beta}^{m} \rangle| \leq \sqrt{\epsilon}\mu_{\infty}^{0}/2\}}  + \delta\|\tilde{\beta}^*-\tilde{\beta}^{m}\|
\end{align*}
for some $\tilde{S}$ that is $s$-sparse. It comes out that
$$
\sqrt{s} \mu_{\infty}^\epsilon  \leq  2\left(1+\frac{1}{\sqrt{\epsilon}}\right)  \delta\|\tilde{\beta}^*-\tilde{\beta}^{m}\|.
$$
Using the induction hypothesis
\begin{align*}
\sqrt{s} \mu_{\infty}^\epsilon \leq  2\sqrt{\delta}(\sqrt{\delta}+1)\|\tilde{\beta}^*-\tilde{\beta}^{m}\|.
\end{align*}
Since $\delta \leq 1/400^2$ and  $p/s$ is large enough the above statement can not hold. Next we have on $S^{c}$
$$
|\tilde{\beta}_{i}^{m+1}| \leq   |\tilde{\Xi}_{i}|\mathbf{1}\{|H^{m+1}_{i}|\geq \mu_{\infty}^\epsilon\} + |\langle \Phi_{i}^{\top},\tilde{\beta}^* - \tilde{\beta}^{m}\rangle|.
$$
Since $S_{m+1}^{c}$ is $s$-sparse then arguing as above we get 
$$
\|\tilde{\beta}^{m+1}_{S^{c}} \| \leq  2\left(1+\frac{1}{\sqrt{\epsilon}}\right)\delta\|\tilde{\beta}^*-\tilde{\beta}^{m}\|.
$$
Hence
$$
\|\tilde{\beta}^{m+1} - \tilde{\beta}^*\| \leq  \sqrt{\delta}(3+4\sqrt{\delta} )\|\tilde{\beta}^*-\tilde{\beta}^{m}\|.
$$
We conclude using the induction hypothesis
since $\sqrt{\delta}(3 + 4 \sqrt{\delta})(10\delta)^{m/2} \leq (10\delta)^{(m+1)/2}$.
As a consequence, as $s/p \to 0$, we have for $m \geq 0$ that 
$$
\|\tilde{\beta}^{m} - \tilde{\beta}^*\| \leq 150 \sqrt{s\log(ep/s)}(10\delta)^{m/2}.
$$
In particular, we have for $m \geq \log(s)$, that 
\[
\| \tilde{\beta}^{m} - \tilde{\beta}^* \|_{\infty} \leq \mu_{\infty}^{\epsilon}/2,
\]
since $\delta$ is small enough. Based on the separation condition on $\tilde{\beta}^*$, it is now clear that $\tilde{\beta}^m$ and $\tilde{\beta}^*$ share the same support. Hence we conclude that
\[
|\tilde{\eta}^m - \eta| = 0.
\]
\end{proof}

\section{Technical lemmas}
\begin{lem}\label{lem:noise_control}
Assume that $\xi$ is a centered $1$-subGaussian random vector. Then, for all $S \subset \{1,\dots,p\}$ such that $|S| \leq s$, we have 
$$
\forall t\geq 0, \quad \mathbf{P}\left( \frac{1}{\|X\|_{2,\infty}^{4}}\sum_{i \in S} \left( X_{i}^{\top}\xi \right)^{2} \geq \frac{s+t}{\|X\|_{2,\infty}^2}\right) \leq e^{-(t/8)\wedge (t^{2}/(64s))}.
$$
\end{lem}
\begin{proof}
Observe that
$$
\sum_{i \in S} \left( X_{i}^{\top}\xi \right)^{2} = \| X_{S}^{\top}\xi \|^{2}.
$$
Then using Theorem $2.1$ in \cite{hsutail} we get that
$$
\mathbf{P}\left( \|X^{\top}_{S}\xi \|^2 \geq Tr(X_{S}X_{S}^{\top}) + 2\|X_{S}X_{S}^{\top} \|_{F}\sqrt{t} + 2\lambda_{\max}(X_{S}X_{S}^{\top})t\right) \leq e^{-t}.
$$
We have that
$$
Tr(X_{S}X_{S}^{\top}) \leq s\|X\|_{2,\infty}^{2}.
$$
Moreover $\lambda_{\max}(X_{S}X_{S}^{\top}) \leq 2\|X\|_{2,\infty}^{2}$ (using RIP) and $\|X_{S}X_{S}^{\top}\|_{F} \leq 2\sqrt{s}\|X\|_{2,\infty}^{2}$ since the rank of $X_{S}X_{S}^{\top}$ is at most $s$. Hence we conclude using the above concentration inequality that
$$
\mathbf{P}\left( \frac{1}{\|X\|_{2,\infty}^{4}}\sum_{i \in S} \left( X_{i}^{\top}\xi \right)^{2} \geq \frac{s+4\sqrt{st}+4t}{\|X\|_{2,\infty}^2}\right) \leq e^{-t}.
$$
The final bound is obtained by considering $u = 8\sqrt{t}(\sqrt{s} \vee \sqrt{t})$ and equivalently $t = (u/8) \wedge (u^{2}/(64 s))$.
\end{proof}
\begin{lem}\label{lem:order:statistic}
There exists $c>0$ such that
$$
 \mathbf{P}\left( \sum_{i=1}^{s} \Xi_{(i)}^{2} \geq \frac{10\sigma^2s\log(ep/s)}{\|X\|^2_{2,\infty}}\right) \leq e^{-cs\log(ep/s)}.
$$
\end{lem}
\begin{proof}
Using Lemma \ref{lem:noise_control}, and the union bound since
$$
\binom{p}{s} \leq \left( \frac{ep}{s}\right)^s\leq e^{s\log(ep/s)}.
$$
The result follows as long as $p \geq s$.
\end{proof}

\begin{lem}\label{lem:subG}
Let $\Xi$ defined as in Section \ref{sec:3}, then we have
$$
\forall t \geq 0, \quad \mathbf{E}(\Xi_{i}^{2}\mathbf{1}\{|\Xi_{i}|\geq t\}) \leq 2\left(t^{2}+\sigma^{2}/\|X\|_{2,\infty}^2\right)e^{-t^{2}\|X\|_{2,\infty}^2/(2\sigma^2)}.
$$
\end{lem}
\begin{proof}
Let $t \geq 0$. We have
\begin{align*}
    \mathbf{E}(\Xi_{i}^{2}\mathbf{1}\{|\Xi_{i}|\geq t\}) &= \int_{0}^{\infty}\mathbf{P}(\Xi_{i}^{2}\mathbf{1}\{|\Xi_{i}|\geq t\} \geq u) du.\\
    &= \int_{0}^{t^{2}}\mathbf{P}(|\Xi_{i}|\geq t) du + \int_{t^{2}}^{\infty}\mathbf{P}(\Xi_{i}^2 \geq u)du \\
    &= t^{2}\mathbf{P}(|\Xi_{i}|\geq t) + 2\int_{t}^{\infty}u\mathbf{P}(|\Xi_{i}| \geq u)du \\
    &\leq 2\left(t^{2}+\sigma^{2}/\|X\|_{2,\infty}^2\right)e^{-t^{2}\|X\|_{2,\infty}^2/(2\sigma^2)}.
\end{align*}

\end{proof}
\begin{lem}\label{lem:barber} Let $\beta$ be a $s$-sparse vector and $S$ its support. Then
\[
\mathbf{P}_{\beta}\left( \underset{\hat\beta}{\sup} \left\langle \xi, \frac{X^{\top}(\beta - \hat{\beta})}{\|X(\beta-\hat{\beta})\|}\right\rangle^2 \geq 7(s + |\hat{\beta}_{S^c}|_0)\log(ep/(s+|\hat{\beta}_{S^c}|_0)) \right) \leq e^{-cs\log(ep/s)}.
\]
\end{lem}
\begin{proof}
Observe that $\beta - \hat{\beta}$ is has sparsity at most $s+|\hat{\beta}_{S^c}|_0$. Let $\pi$ be the orthogonal projector onto the span of columns of $X$ indexed by the support of $\beta - \hat{\beta}$. It comes out that
\[
\left|\left\langle \xi, \frac{X^{\top}(\beta - \hat{\beta})}{\|X(\beta-\hat{\beta})\|}\right\rangle\right| \leq \|\pi \xi\|.
\]
Since $\pi$ has rank at most $s+|\hat{\beta}_{S^c}|_0$, then using Theorem $2.1$ in \cite{hsutail} we get that for all $t \geq 0$ and for fixed $\pi$
$$
\mathbf{P}\left( \|\pi\xi \|^2 \geq s+|\hat{\beta}_{S^c}|_0 + 3t\right) \leq e^{-t}.
$$
It comes out that
\[
\mathbf{P}\left( \|\pi\xi \|^2 \geq 7(s+|\hat{\beta}_{S^c}|_0)\log(ep/(s+|\hat{\beta}_{S^c}|))\right) \leq e^{-2s\log(ep/s)}.
\]
Hence
\[
\mathbf{P}_\beta\left( \underset{|\hat\beta_{S^c}|_0 = v}{\sup}\|\pi\xi \|^2 \geq 7(s+v)\log(ep/(s+v))\right) \leq e^{-s\log(ep/s)}.
\]
We conclude using a union bound over all values of $v$.
\end{proof}
\end{document}